\documentclass[AMA,STIX1COL]{WileyNJD-v2}

\articletype{}\received{26 April 2016}
\revised{6 June 2016}
\accepted{6 June 2016}

\usepackage{bm}

\renewcommand{\epsilon}{\varepsilon}
\renewcommand{\phi}{\varphi}

\newcommand{\sign}{\operatorname{sign}}
\newcommand{\abs}[1]{\left\lvert #1 \right\rvert}

\newcommand{\TT}{^{\mathrm{T}}}

\newcommand{\diffd}{\mathrm{d}}
\newcommand{\dt}[1]{\deriv{#1}{t}}
\newcommand{\deriv}[2]{\frac{\diffd {#1}}{\diffd  {#2}}}
\newcommand{\derivk}[3]{\frac{\diffd^{#3} {#1}}{\diffd  {#2}^{#3}}}

\newcommand{\spow}[2]{\left\lfloor #1 \right\rceil^{#2}}
\newcommand{\spowf}[3]{\spow{#1}{\frac{#2}{#3}}}
\newcommand{\apow}[2]{\abs{#1}^{#2}}

\def\F{\mathbf{F}}

\def\P{\mathbf{P}}

\def\S{\mathbf{S}}

\def\g{\mathbf{g}}

\def\x{\mathbf{x}}
\def\y{\mathbf{y}}
\def\z{\mathbf{z}}

\newcommand{\vxi}{\bm{\xi}}

\newcommand{\RR}{\mathbb{R}}

\newcommand{\NN}{\mathbb{N}}
\newcommand{\ZZ}{\mathbb{Z}}

\usepackage{tikz}
\usepackage[export]{adjustbox}
\usepackage{diagbox}

\newtheorem{thm}{Theorem}[section]
\newtheorem{prop}[thm]{Proposition}
\newtheorem{cor}[thm]{Corollary}
\newtheorem{lem}[thm]{Lemma}
\theoremstyle{definition}
\newtheorem{rem}[thm]{Remark}
\newtheorem{defn}[thm]{Definition}
\newtheorem{exmp}[thm]{Example}

\begin{document}

\title{Proper Implicit Discretization of Arbitrary-Order Robust Exact Differentiators}

\author[1]{Richard Seeber}

\authormark{RICHARD SEEBER}

\address[1]{\orgdiv{Christian Doppler Laboratory for Model Based Control of Complex Test Bed Systems, Institute of Automation and Control}, \orgname{Graz University of Technology}, \orgaddress{\city{Graz}, \country{Austria}}}

\corres{*Corresponding author: R. Seeber. \email{richard.seeber@tugraz.at}}

\abstract[Summary]{This paper considers the implicit Euler discretization of Levant's arbitrary order robust exact differentiator in presence of sampled measurements.
    Existing implicit discretizations of that differentiator are shown to exhibit either unbounded bias errors or, surprisingly, discretization chattering despite the use of the implicit discretization.
    A new, proper implicit discretization that exhibits neither of these two detrimental effects is proposed  by computing the differentiator's outputs as appropriately designed linear combinations of its state variables.
    A numerical differentiator implementation is discussed and closed-form stability conditions for arbitrary differentiation orders are given. The influence of bounded measurement noise and numerical approximation errors is formally analyzed.
    Numerical simulations confirm the obtained results.
}

\keywords{Levant's differentiator, implicit discretization, sliding mode observer, robust exact differentiator, measurement noise, sampled measurements}
\maketitle

\newcommand{\FL}[1]{\mathcal{F}_L^{(#1)}}
\newcommand{\FLb}[1]{\mathcal{F}_{M}^{(#1)}}
\newcommand{\FLz}[1]{\mathcal{F}_{0}^{(#1)}}
\newcommand{\EN}{\mathcal{E}_N}
\newcommand{\Diff}[1]{\mathcal{D}^{(#1)}_T}
\newcommand{\Diffh}[1]{\hat{\mathcal{D}}^{(#1)}_T}

\section{Introduction}

The differentiation of measured signals is an important task in many control related engineering applications.
Accordingly, many approaches for this purpose exist in literature.
Among the most important ones are linear high-gain differentiators\cite{khalil2002nonlinear,vasiljevic2008error}, linear algebraic differentiators\cite{mboup2018frequency}, and robust exact differentiators based on sliding modes\cite{levant1998robust,levant_ijc03}.
The latter have the particularly attractive feature that they differentiate exactly in the absence of measurement noise while also being robust in its presence.

In practice, measurements are typically available only in the form of sampled signals.
For this case, different discrete-time implementations of Levant's robust exact differentiator\cite{levant_ijc03} (RED) exist.
Its proper explicit (forward Euler) discretization proposed by Livne and Levant\cite{livne2014proper} preserves the asymptotic accuracies of the continuous-time differentiator, but suffers from discretization chattering that increases the differentiation error.
The implicit discretization technique, originally proposed by Acary and Brogliato\cite{acary2010implicit,brogliato2021digital} and later applied to sliding mode differentiators by Mojallizadeh et al.\cite{mojallizadeh2021time} and Carvajal-Rubio et al.\cite{carvajal2021implicit,carvajal2022implicit}, is capable in theory to avoid this type of chattering and to yield quasi-exact sample-based differentiators as introduced by Seeber and Haimovich\cite{seeber202Xoptimal}.

A number of implicitly discretized variants of the robust exact differentiator have been proposed in literature.
Mojallizadeh et al.\cite{mojallizadeh2021time} obtain the so-called implicit homogeneous discrete-time differentiator (I-HDD) by means of a straightforward modification of expressions of the properly discretized explicit Euler discretization.
While this approach indeed eliminates chattering, it exhibits a possibly unbounded bias error starting with differentiation order two.
Carvajal-Rubio et al.\cite{carvajal2021implicit} avoid this bias error in their proposed homogeneous implicit discrete-time differentiator (HIDD) by designing an observer for a discrete-time integrator chain, rather than modifying an existing (explicit) discretization of the differentiator.
However, their approach exhibits a chattering of the differentiation error that occurs, surprisingly, despite the use of the implicit discretization.
For that reason, structural conditions for obtaining a proper implicit discretization within have recently been explored by Seeber and Koch\cite{seeber2023structural}.

The present paper proposes a proper implicit Euler discretization of the robust exact differentiator---called the implicit robust exact differentiator (IRED)---which, in contrast to existing approaches, exhibits neither \emph{discretization chattering} nor \emph{bias errors}.
Compared to the recently proposed structural conditions\cite{seeber2023structural} for attaining these properties within a class of implicitly discretized differentiators, it tackles the problem from a different angle:
by computing the differentiator's outputs by means of appropriately designed linear combinations of its state variables, rather than using the state variables themselves as outputs, as it is usual in all existing differentiator structures.
This approach yields a proper implicit discretization of the robust exact differentiator in an appealingly simple form.
Based on a Lyapunov function recently proposed for the continuous-time differentiator\cite{seeber2023closed}, the proposed approach constitutes, moreover, the first \emph{discrete-time} implementation of the robust exact differentiator that is accompanied by \emph{closed-form} stability conditions for \emph{arbitrary} differentiation orders.
The URL \url{https://github.com/seeberr/ired_toolbox} provides a toolbox implementation of the proposed IRED for Matlab/Simulink.\cite{seeber2024implicit}

The paper is structured as follows.
Section~\ref{sec:problem} introduces the considered problem of numerical signal differentiation from sampled measurements and motivates the present paper by showing that existing implicit discretizations of the robust exact differentiator exhibit either discretization chattering or bias errors.
Section~\ref{sec:proposed} then introduces the proposed implicit robust exact differentiator  and states the main results: closed-form sufficient conditions for Lyapunov stability and finite-time convergence for arbitrary differentiation orders, bounds on the differentiation error showing robustness to measurement noise and absence of chattering, and an analysis of the influence of approximation errors in the differentiator's numerical implementation.
Section~\ref{sec:stability} then performs the formal stability analysis to prove the theorems from Section~\ref{sec:proposed}.
Section~\ref{sec:simulation} illustrates the proposed differentiator's performance and compares it to existing approaches from literature in a simulation example.
Section~\ref{sec:conclusion} gives concluding remarks.
Proofs of all lemmata are given in Appendix~\ref{app:proofs}.

\textbf{Notation:}
The sets $\RR$, $\RR_{\ge 0}$, $\RR_{> 0}$ denote the reals, nonnegative reals, and positive reals, respectively, $\NN$ and $\NN_0$ are the positive and nonnegative integers, and vectors are written as boldface lowercase letters.
For a function $f : \RR\to \RR$, $\dot f = \dt{f}$ and $\ddot f = \derivk{f}{t}{2}$ denote its first and second time derivative, and $f^{(i)} = \derivk{f}{t}{i}$ is written for its $i$th time derivative in general.
For $y, p \in \RR$ with $p \ne 0$, the  abbreviation $\spow{y}{p} = \apow{y}{p} \sign(y)$ is used.
The abbreviation $\spow{y}{0}$ denotes the \emph{set-valued} sign function defined as $\spow{y}{0} = \{ \sign(y) \}$ for $y \ne 0$ and $\spow{0}{0} = [-1, 1]$.
For a set $\mathcal{M} \subseteq \RR$ and a scalar $a \in \RR$, addition and multiplication involving the set are defined as $a + \mathcal{M} = \{ a + x : x \in \mathcal{M} \}$ and $a \mathcal{M} = \{ ax : x \in \mathcal{M} \}$.
For integers $i,j \in \NN_0$ with $0 \le j \le i$, the binomial coefficient is written as $\binom{i}{j} = \frac{i! }{j!(i-j)!}$, where $i! = \prod_{k=1}^{i} k$ denotes the factorial of $i$.

\section{Problem Statement}
\label{sec:problem}

\subsection{Continuous-Time Robust Exact Differentiation}

Consider an $m$ times differentiable signal $f : \RR_{\ge 0} \to \RR$, whose $m$th derivative $f^{(m)}$ is globally Lipschitz continuous with Lipschitz constant $L \in \RR_{\ge 0}$.
Its $m$ derivatives  may then be obtained by means of Levant's robust exact differentiator\cite{levant_ijc03}
\begin{equation}
    \label{eq:diff:levant}
    \begin{aligned}
        \dot z_1 &= z_2 + \lambda_1 L^{\frac{1}{m+1}} \spowf{f - z_1}{m}{m+1}, \\
&\vdots \\
        \dot z_{m} &= z_m + \lambda_{m} L^{\frac{m}{m+1}} \spowf{f - z_1}{1}{m+1}, \\
        \dot z_{m+1} &= \lambda_{m+1} L \sign(f - z_1)
    \end{aligned}
\end{equation}
with positive parameters $\lambda_1, \ldots, \lambda_{m+1}$, outputs $y_1 = z_2, \ldots, y_m = z_{m+1}$, and solutions understood in the sense of Filippov\cite{filippov1988differential}.
In absence of measurement noises and for appropriately selected parameters, the outputs $y_1, \ldots, y_m$ exactly reconstruct the derivatives $\dot f, \ldots, f^{(m)}$ after a finite convergence time, i.e., there exists a time $\tau$ depending on the initial conditions of the differentiator and of the signal's derivatives, such that $y_i(t) = f^{(i)}(t)$ holds for all $t \ge \tau$ and all $i = 1, \ldots, m$.
Moreover, the differentiator is robust to additive measurement noise with uniform bound $N \in \RR_{\ge 0}$ in the sense that in such case the differentiation error of the $i$th derivative is ultimately bounded by $a_i N^{\frac{m+1-i}{m+1}} L^{\frac{i}{m+1}}$ for some constants $a_1, \ldots, a_m \in \RR_{> 0}$ depending only on the parameters $\lambda_1, \ldots, \lambda_{m+1}$.

\subsection{Signal Differentiation from Sampled and Noisy Measurements}

Now and for the remainder of the paper, consider the case that the signal $f : \RR_{\ge 0} \to \RR$ to be differentiated is sampled with sampling period $T \in \RR_{>0}$ and is corrupted by additive measurement noise bounded by $N \in \RR_{\ge 0}$.
The sampled and noisy measurements of $f$ are denoted by $u_k = f(kT) + \eta_k$ ($k=0,1,2, \ldots$) with a measurement noise sequence $(\eta_k)$.
As before, the $m$th derivative of $f$ is assumed to be Lipschitz continuous.
Consequently, the $(m+1)$th derivative $f^{(m+1)}$ of the signal and the noise sequence $(\eta_k)$ are assumed to satisfy
\begin{equation}
    \label{eq:constraints}
    |f^{(m+1)}(t)| \le L \qquad\text{and}\qquad
    |\eta_k| \le N
\end{equation}
almost everywhere on $\RR_{\ge 0}$ and for all $k \in \NN_0$, respectively, with known Lipschitz constant $L \in \RR_{> 0}$ and unknown noise bound $N \in \RR_{\ge 0}$.

Denote by $\mathcal{S}$ the set of sequences with values in $\RR$.
A sample-based $m$th order differentiator with sampling time $T$ is then understood to be a causal operator $\Diff{m} : \mathcal{S}  \to \mathcal{S}^m$ mapping the measurement sequence $(u_k)$ to $m$ output sequences $[(y_{1,k}), \ldots, (y_{m,k})]\TT = \Diff{m}[ (u_k) ]$, where $y_{i,k}$ constitutes an estimate for the $i$th time derivative of $f$ at time $kT$, i.e., for $f^{(i)}(kT)$.
The elements of the output sequence are aggregated in the vector $\y_k = [y_{1,k} \, \ldots \, y_{m,k}]\TT$ and $(\y_k) = \Diff{m}[ (u_k) ]$ is written with slight abuse of notation.
Moreover, the set of all admissible signals and noise sequences satisfying \eqref{eq:constraints} is denoted by $\FL{m}$ and $\EN$, respectively, i.e.,
\begin{align}
    \FL{m} &= \{ f \in \mathcal{F}^{(m)} : |f^{(m+1)}(t)| \le L \text{ almost everywhere on $\RR_{\ge 0}$} \} \\
    \EN &= \{ (\eta_k) \in \mathcal{S} : |\eta_k| \le N \text{ for all $k$} \}
\end{align}
wherein $\mathcal{F}^{(m)}$ is the set of all $m$ times differentiable functions $f : \RR_{\ge 0}\to \RR$ whose $m$th derivative is Lipschitz continuous.
In the following, sample based differentiators are considered that are obtained from Levant's robust exact differentiator \eqref{eq:diff:levant} by means of an implicit Euler discretization.

As a subclass of sample-based differentiators, a sample-based implicit sliding-mode differentiator with sampling time $T$ is understood to be a state-space system of the form
\begin{equation}
    \label{eq:general:diff}
    \z_{k+1} \in \F_T(\z_k, \z_{k+1}, u_k), \qquad
    \y_k = \g_T(\z_k, \z_{k+1}, u_k)
\end{equation}
with fixed initial values.
Therein, the set-valued function $\F_T : \RR^n \times \RR^n \times \RR \to 2^{\RR^n}$ denotes the set-valued right-hand side of the implicit difference inclusion that is typically obtained when applying an implicit discretization technique\cite{brogliato2021digital}.
The function $\g_T : \RR^n \times \RR^n \times \RR \to \RR^m$ denotes the output map that maps the states and possibly the input to the differentiator's outputs.
It is assumed that $\F$ is such that for every initial condition $\z_1$ and input sequence $(u_k) \in \mathcal{S}$, \eqref{eq:general:diff} has a well-defined, unique solution sequence $(\z_k)$ and corresponding output sequence $(\y_k)$.
The next definition characterizes what is to be understood as a discrete-time sliding motion of such a system.
\begin{defn}
    \label{def:slidingmode}
    The sample-based implicit sliding-mode differentiator $\Diff{m}$ in the form \eqref{eq:general:diff} with given initial condition $\z_1$ and input sequence $(u_k)$ is said to be \emph{in discrete-time sliding mode} at time index $k$ if the set $\F_T(\z_k, \z_{k+1}, u_k)$ contains more than one element, i.e., if $\F_T(\z_k, \z_{k+1}, u_k) \setminus \{ \z_{k+1} \}$ is non-empty.
\end{defn}

\subsection{Motivating Examples -- Chattering and Bias of Existing Implicit Differentiators}

The implicit discretization technique is well known for its capability to avoid chattering of discrete-time implementations of sliding mode systems.\cite{brogliato2021digital}
Its application for discretizing the arbitrary-order robust exact differentiator is studied in the works of Carvajal-Rubio et al.\cite{carvajal2021implicit} and Mojallizadeh et al.\cite{mojallizadeh2021time}
As the following two examples show, the implicitly discretized differentiators proposed and studied therein may still exhibit chattering or---in the latter case---even unbounded bias errors, however.

\begin{exmp}
    \label{ex:hidd}
    Consider the first-order homogeneous implicit discrete-time differentiator (HIDD) proposed by Carvajal-Rubio et al.\cite{carvajal2021implicit}, which may be written as
    \begin{subequations}
        \label{eq:diff:carvajal-rubio}
        \begin{align}
            z_{1,k+1} &= z_{1,k} + T z_{2,k} + T \lambda_1 L^{\frac{1}{2}} \spowf{u_k - z_{1,k+1}}{1}{2} + \frac{\lambda_2 L T^2}{2} \xi_k  \\
            z_{2,k+1} &= z_{2,k} + T \lambda_2 L \xi_k \\
            \xi_k &\in \spow{u_k - z_{1,k+1}}{0}
        \end{align}
    \end{subequations}
    with output $y_{1,k} = z_{2,k+1}$.
    With regard to Definition~\ref{def:slidingmode}, this differentiator is in discrete-time sliding mode at time index $k$ if $z_{1,k+1} = u_k$, because then the set $\spow{u_k - z_{1,k+1}}{0} = \spow{0}{0} = [-1, 1]$ contains more than one element.
    Consider differentiation of the signal $f(t) = \lambda_2 L T t/2$ in the absence of noise with initial conditions $z_{1,1} = z_{2,1} = 0$.
    Note that $f \in \FL{1}$ holds regardless of $\lambda_2$, because $\ddot f(t) = 0$.
    With $u_k = f(kT) =  \lambda_2 L T^2 k/2$, it is easy to verify by substitution into \eqref{eq:diff:carvajal-rubio} that the corresponding solution is given by $z_{1,k+1} = u_k$ and
    \begin{equation}
        y_{1,k} = z_{2,k+1} = [1 - (-1)^k] \frac{\lambda_2 L T}{2}.
\end{equation}
    Obviously, this solution exhibits a chattering differentiation error $y_{1,k} - \dot f(kT) = \frac{\lambda_2 L T}{2} (-1)^{k+1} $.
\end{exmp}

\begin{exmp}
    \label{ex:ihdd}
    Consider the following family of sample-based second-order differentiators
    \begin{subequations}
        \label{eq:ihdd:family}
        \begin{align}
            z_{1,k+1} &= z_{1,k} + T \lambda_1 L^{\frac{1}{3}} \spowf{u_k - z_{1,k+1}}{2}{3} + T z_{2,k+1} + c \frac{T^2}{2} z_{3,k+1} \\
            z_{2,k+1} &= z_{2,k} + T \lambda_2 L^{\frac{2}{3}} \spowf{u_k - z_{1,k+1}}{1}{3} + T z_{3,k+1} \\
            z_{3,k+1} &\in z_{3,k} + T \lambda_3 L \spow{u_k-z_{1,k+1}}{0}
        \end{align}
    \end{subequations}
    with constant parameter $c \in \RR$ and outputs $y_{1,k} = z_{2,k+1}$, $y_{2,k} = z_{3,k+1}$.
    Note that for $c = 1$ and $c = 0$, respectively, the differentiator \eqref{eq:ihdd:family} corresponds to the implicit homogeneous discrete-time differentiator (I-HDD) and the implicit arbitrary-order super-twisting differentiator (I-AO-STD) studied by Mojallizadeh et al\cite{mojallizadeh2021time}.
    Consider now  differentiation of the signal $f(t) = \alpha t^2$ with arbitrary $\alpha \in \RR_{\ge 0}$ in absence of noise, leading to the measurement sequence $u_k = \alpha k^2 T^2$.
    Obviously, $f \in \FLz{2}$ for all $\alpha \ge 0$.
    One may verify that a solution of \eqref{eq:ihdd:family} is given by $z_{1,k+1} = u_k$ and
    \begin{equation}
        y_{1,k} = z_{2,k+1} = 2 \alpha k T - (1+c) \alpha T, \qquad
        y_{2,k} = z_{3,k+1} = 2\alpha.
    \end{equation}
    Obviously, the error $|y_{1,k} - \dot f(kT)| = |1+c| \alpha T$ is an unbounded function of $\alpha$, unless $c = -1$, which hence is the only differentiator among the family \eqref{eq:ihdd:family} that does not exhibit a bias error.
    For $c = -1$, however, another family of solutions is given by $z_{1,k+1} = u_k$ and
    \begin{equation}
        y_{1,k} = z_{2,k+1} = 2 \alpha k T + \gamma \frac{\lambda_3 L T^2}{4} (-1)^k, \qquad
        y_{2,k} = z_{3,k+1} = 2\alpha + \gamma \frac{\lambda_3 L T}{2} (-1)^k
    \end{equation}
    with arbitrary $\gamma \in [-1,1]$, leading to chattering differentiation errors $y_{1,k} - \dot f(kT) = \gamma \frac{\lambda_3 L T^2}{4} (-1)^k$, $y_{2,k} - \ddot f(kT) = \gamma \frac{\lambda_3 L T}{2} (-1)^k$ in general.
    Hence, both the I-HDD ($c = 1$) and the I-AO-STD ($c = 0$) exhibit a bias error.
    There also exists no alternative implicit implementation in the form \eqref{eq:ihdd:family} that exhibits neither bias errors nor chattering.
\end{exmp}

The previous example shows that the remedy proposed by Livne and Levant\cite{livne2014proper} for the forward Euler discretization---adding appropriate Taylor terms in the difference equations---does not work in the case of the implicit discretization without possibly encouraging chattering.
Hence, a different approach is required to obtain a proper implicit discretization of the robust exact differentiator, as discussed in the next section.

\subsection{Proper Implicit Discretization of the Robust Exact Differentiator}

In order to formally characterize the requirement for absence of the discretization chattering and bias errors that are present in the presented examples, the notion of a proper implicit discretization of \eqref{eq:diff:levant} was introduced by Seeber and Koch.\cite[Definition~1]{seeber2023structural}
Here, this notion is extended to the considered class of sample-based implicit sliding-mode differentiators in the form \eqref{eq:general:diff}.
In particular, a proper implicit discretization is defined by requiring that certain differentiation error bounds are eventually established once the differentiator is in discrete-time sliding mode.
\begin{defn}
    \label{def:proper}
    A sample-based implicit sliding-mode differentiator $\Diff{m}$ in the form \eqref{eq:general:diff} with sampling time $T \in \RR_{> 0}$ is said to be a \emph{proper implicit discretization} for signals in $\FL{m}$, if there exist constants $\mu_1, \ldots, \mu_m$ such that for all $M \in [0, L]$, $K \in \NN$, and for every (noise-free) input sequence $u_k = f(kT)$ with $f \in \FLb{m}$, the following implication is true:
    If the differentiator with input sequence $(u_k)$ is in discrete-time sliding mode for all $k \ge K$, then there exists an integer $\bar K \ge K$ such that the differentiator's output sequence $(\y_k) = \Diff{m}[ (u_k) ]$ fulfills
    \begin{equation}
        |y_{i,k} - f^{(i)}(kT)| \le \mu_i M T^{m+1-i}
    \end{equation}
    for all $k \ge \bar K$.
\end{defn}

By setting $M = 0$ in the above definition, it is obvious that a proper implicit discretization of an $m$th order robust exact differentiator, in particular, \emph{exactly} differentiates all polynomials with degree up to $m$ after a \emph{finite time}, provided that the sliding mode is attained in finite time.
As a consequence, neither bias nor chattering is present in such cases.
The present paper proposes a new implicit discretization of the arbitrary order robust exact differentiator that
\begin{itemize}
    \item 
        is a proper implicit discretization, and thus exhibits neither discretization chattering nor bias errors,
    \item
        is accompanied by closed-form stability conditions and differentiation error bounds,
    \item
        and is proven to converge in finite time subject to these conditions.
\end{itemize}

\section{Proposed Differentiator and Main Results}
\label{sec:proposed}

\subsection{Implicit Robust Exact Differentiator (IRED)}

In implicit form, the proposed $m$th order implicit robust exact differentiator (IRED) is given by
\begin{subequations}
    \label{eq:diff}
\begin{align}
    \label{eq:diff:zkp1i}
    z_{i,k+1} &= z_{i,k} + T \lambda_i L^{\frac{i}{m+1}} \spowf{u_k - z_{1,k+1}}{m-i+1}{m+1} + T z_{i+1,k+1} \qquad \text{for $i = 1, \ldots, m$} \\
    \label{eq:diff:zkp1n}
    z_{m+1,k+1} &\in z_{m+1,k} + T \lambda_{m+1} L \spow{u_k-z_{1,k+1}}{0}
    \intertext{with outputs}
    \label{eq:diff:y}
    y_{i,k} &= \sum_{j=i}^{m} T^{j-i} c_{i,j} z_{j+1,k+1} \qquad \text{for $i = 1, \ldots, m$} \end{align}
\end{subequations}
where the constants $c_{i,j}$ for $i,j \in \NN$ are recursively defined as
\begin{equation}
    \label{eq:cij}
    c_{i,j} = \frac{(j-1) c_{i,j-1} + i c_{i-1,j-1}}{j}
\end{equation}
with initial values $c_{0,0} = 1$ and $c_{0,j} = c_{i,0} = 0$ for $i,j\ne 0$.
Note that, for all $j \in \NN_0$, $c_{j,j} = 1$ and $c_{i,j} = 0$ holds for $i > j$.
Table~\ref{tab:cij} lists the values of the constants $c_{i,j}$ that are relevant for differentiator orders $m \le 6$.

\begin{rem}
    The crucial difference between the proposed IRED \eqref{eq:diff} and the I-AO-STD\cite{mojallizadeh2021time} is that the outputs of the former are not the state variables themselves, but rather appropriate linear combinations thereof.
    As will be shown, this yields a sample-based implicit sliding-mode differentiator without the drawbacks of the I-AO-STD (bias error) or of the HIDD\cite{carvajal2021implicit} and I-HDD\cite{mojallizadeh2021time} (discretization chattering).
\end{rem}

\begin{rem}
    The proposed approach can also be used to obtain a proper implicit discretization of the robust exact filtering differentiator proposed by Levant and Livne\cite{levant2020robust}.
Denoting the filtering order by $q \in \NN$ and the differentiation order by $m\in \NN$, the implicit robust exact filtering differentiator of order $m$ and filtering order $q$ is given by
\begin{subequations}
    \label{eq:diff:filt}
\begin{align}
    \label{eq:diff:filt:zkp1i}
    z_{i,k+1} &= z_{i,k} - T \lambda_i L^{\frac{i}{q+m+1}} \spowf{z_{1,k+1}}{q+m-i+1}{q+m+1} + T z_{i+1,k+1} \qquad \text{for $i = 1, \ldots, q-1, q+1, \ldots, q+m$} \\
    z_{q,k+1} &= z_{q,k} - T \lambda_q L^{\frac{q}{q+m+1}} \spowf{z_{1,k+1}}{m+1}{q+m+1} + T z_{q+1,k+1} - T u_k \\
    \label{eq:diff:filt:zkp1n}
    z_{q+m+1,k+1} &\in z_{q+m+1,k} - T \lambda_{q+m+1} L \spow{z_{1,k+1}}{0}
    \intertext{with outputs}
    \label{eq:diff:filt:y}
    y_{i,k} &= \sum_{j=i}^{m} T^{j-i} c_{i,j} z_{q+j+1,k+1} \qquad \text{for $i = 1, \ldots, m$.} \end{align}
\end{subequations}
All following results and proofs may be extended to this differentiator by means of straightforward modifications.
\end{rem}

\begin{table}
    \centering
    \renewcommand{\arraystretch}{1.4}\begin{tabular}{|c||c|c|c|c|c|c|c|}
        \hline
        \diagbox{$i$}{$j$} & 1 & 2 & 3 & 4 & 5 & 6 & 7 \\\hline\hline
        1 & $1$ & $\frac{1}{2}$ & $\frac{1}{3}$ & $\frac{1}{4}$ & $\frac{1}{5}$ & $\frac{1}{6}$ & $\frac{1}{7}$ \\
        \hline
        2 & $0$ & $1$ & $1$ & $\frac{11}{12}$ & $\frac{5}{6}$ & $\frac{137}{180}$ & $\frac{7}{10}$ \\
        \hline
        3 & $0$ & $0$ & $1$ & $\frac{3}{2}$ & $\frac{7}{4}$ & $\frac{15}{8}$ & $\frac{29}{15}$ \\
        \hline
        4 & $0$ & $0$ & $0$ & $1$ & $2$ & $\frac{17}{6}$ & $\frac{7}{2}$ \\
        \hline
        5 & $0$ & $0$ & $0$ & $0$ & $1$ & $\frac{5}{2}$ & $\frac{25}{6}$ \\
        \hline
        6 & $0$ & $0$ & $0$ & $0$ & $0$ & $1$ & $3$ \\
        \hline
    \end{tabular}
    \caption{Coefficients $c_{i,j}$ as defined in \eqref{eq:cij} for the IRED with differentiator order $m \le 6$}
    \label{tab:cij}
\end{table}

\subsection{Numerical Implementation}
\label{sec:implementation}

To obtain a numerical implementation of \eqref{eq:diff} in explicit form, obtain\cite{mojallizadeh2021time,carvajal2021implicit} by successive substitution from \eqref{eq:diff:zkp1i}--\eqref{eq:diff:zkp1n} the generalized equation
\begin{equation}
    z_{1,k+1} \in \lambda_{m+1} L T^{m+1} \spow{u_k - z_{1,k+1}}{0} + \sum_{i=1}^{m} T^i \lambda_i L^{\frac{i}{m+1}} \spowf{u_k - z_{1,k+1}}{m-i+1}{m+1} + \sum_{i=1}^{m+1}T^{i-1}z_{i,k}.
\end{equation}
From this relation, one may verify that $\rho_k = L^{-\frac{1}{m+1}} T^{-1} \spowf{u_k - z_{1,k+1}}{1}{m+1}$ satisfies the generalized equation
\begin{equation}
    \label{eq:alpha:inclusion}
0 \in \left( \spow{\rho_k}{m+1} + \lambda_1 \spow{\rho_k}{m} + \ldots + \lambda_{m} \rho_k + \lambda_{m+1} \spow{\rho_k}{0} \right) LT^{m+1} - b_k
\end{equation}
with
\begin{equation}
    \label{eq:bk}
    b_k = u_k - \sum_{i=1}^{m+1} T^{i-1} z_{i,k}.
\end{equation}
In case $|b_k| \le \lambda_{m+1} L T^{m+1}$, this generalized equation implies that $\rho_k = 0$.
Otherwise, $\rho_k = r_k \sign(b_k)$ where $r_k = |\rho_k|$ is a positive solution of the polynomial equation
\begin{equation}
    \label{eq:alpha:poly}
     r_k^{m+1} + \lambda_1 r_k^{m} + \ldots + \lambda_{m} r_k + \lambda_{m+1}  - \frac{|b_k|}{LT^{m+1}} = 0.
\end{equation}
For $\lambda_1, \ldots, \lambda_{m+1} \in \RR_{\ge 0}$ and $|b_k| > \lambda_{m+1} L T^{m+1}$, the left hand side of this equation is negative for $r_k = 0$ and is strictly increasing and unbounded for $r_k \to \infty$.
Hence, such a solution always exists and is unique.

In a numerical implementation, Newton's method or other root finding algorithms\cite{carvajal2021implicit,mojallizadeh2021time} may be used to approximate $r_k$.
In such a case, \eqref{eq:alpha:poly} is usually only satisfied up to some residual error, i.e.,
\begin{equation}
    \label{eq:alpha:poly:approx}
    \hat r_k^{m+1} + \sum_{i=1}^{m+1} \lambda_i \hat r_k^{m-i+1}  - \frac{|b_k|}{LT^{m+1}}  \in [-R,R]
\end{equation}
holds for the numerical approximation $\hat r_k$ of $r_k$, where the constant $R \in \RR_{> 0}$ is  a tuning parameter of the root finding method.
With $\hat r_k > 0$ satisfying \eqref{eq:alpha:poly:approx}, a numerical implementation then may compute an approximation $\hat \rho_k \approx L^{-\frac{1}{m+1}} T^{-1} \spowf{u_k - z_{1,k+1}}{1}{m+1}$ for $\rho_k$ as well as the state update \eqref{eq:diff:zkp1i}--\eqref{eq:diff:zkp1n} in decreasing order of the state variables as
\begin{subequations}
    \label{eq:diff:approx}
    \begin{align}
        \hat \rho_k &= \begin{cases}
            0 & \text{if $|b_k| \le \lambda_{m+1} L T^{m+1}$} \\
            \hat r_k \sign(b_k) & \text{otherwise},
        \end{cases} \\
        z_{m+1,k+1} &= \begin{cases}
            z_{m+1,k} + \frac{b_k}{T^{m}} & \text{if $|b_k| \le \lambda_{m+1} L T^{m+1}$} \\
            z_{m+1,k} + \lambda_{m+1} LT \sign(b_k) & \text{otherwise}
        \end{cases}, \\
        z_{i,k+1} &= z_{i,k} + T z_{i+1,k+1} + \lambda_i L T^{m-i+2} \spow{\hat \rho_k}{m-i+1}, \qquad \text{for $i = m, \ldots, 1$}.
    \end{align}
\end{subequations}
Finally, differentiator outputs $y_{1,k}, \ldots, y_{m,k}$ may be computed according to \eqref{eq:diff:y}.

The following proposition shows that this approximate numerical implementation behaves like an ideal implementation with additional noise of magnitude at most $R$ added to the measurements $u_k$.
\begin{prop}
    \label{prop:approximation}
    Let $m \in \NN$ and $L, R, T, \lambda_1, \ldots, \lambda_{m+1} \in \RR_{> 0}$.
    Consider the sample-based sliding-mode differentiator $\Diff{m}$ defined in \eqref{eq:diff} and its numerical implementation $\Diffh{m}$ with identical initial condition and output equations \eqref{eq:diff:y} but with state update according to \eqref{eq:diff:approx} wherein $\hat r_k \in \RR_{> 0}$ satisfies \eqref{eq:alpha:poly:approx} and $b_k$ is defined in \eqref{eq:bk}.
    Then, for every sequence $(u_k)$, there exists a sequence $(\epsilon_k)$ satisfying $|\epsilon_k| \le R L T^{m+1}$ such that the outputs of the two differentiators satisfy $\Diffh{m}[(u_k)]  = \Diff{m}[(u_k + \epsilon_k)]$.
    Moreover, $\epsilon_k = 0$ holds for all $k \in \NN$ where $\Diff{m}$ is in discrete-time sliding mode or, equivalently, where $|b_k| \le \lambda_{m+1} L T^{m+1}$.
\end{prop}
\begin{proof}
    For each $k \in \NN$, define $\hat b_k = b_k + \epsilon_k$ and distinguish the cases $|b_k| \le \lambda_{m+1} L T^{m+1}$ and $|b_k| > \lambda_{m+1} L T^{m+1}$.
    In the first case, let $\epsilon_k = 0$.
    In the second case, due to \eqref{eq:alpha:poly:approx}, there exists $\epsilon_k \in [-RLT^{m+1},RL T^{m+1}]$ such that
    \begin{equation}
        \label{eq:akhat:poly}
         \hat r_k^{m+1} + \sum_{i=1}^{m+1} \lambda_i \hat r_k^{m-i+1}  - \frac{|\hat b_k|}{LT^{m+1}}   = 0
    \end{equation}
    holds for $\hat b_k = b_k + \epsilon_k$.
    Additionally, $|\hat b_k| > \lambda_{m+1} L T^{m+1} \hat r_k^0 = \lambda_{m+1} L T^{m+1}$ holds in such case, because $\hat r_k > 0$.
With the sequence $(\epsilon_k)$ thus defined, let $\hat u_k = u_k + \epsilon_k$ and note that
    \begin{equation}
        \label{eq:bkhat}
        \hat b_k = b_k + \epsilon_k = u_k  + \epsilon_k - \sum_{i=1}^{m+1} T^{i-1} z_{i,k} = \hat u_k - \sum_{i=1}^{m+1} T^{i-1} z_{i,k}
    \end{equation}
	is obtained from \eqref{eq:bk}.
    Moreover, $|b_k| > \lambda_{m+1} L T^{m+1}$ holds if and only if $|\hat b_k| > \lambda_{m+1} L T^{m+1}$.
    The claim $\Diffh{m}[(u_k)] = \Diff{m}[(\hat u_k)]$ can then be seen to be true by comparing the relations \eqref{eq:alpha:poly}, \eqref{eq:bk}, which are satisfied by an ideal implementation of $\Diff{m}$, with the relations \eqref{eq:akhat:poly}, \eqref{eq:bkhat} holding for the approximate implementation $\Diffh{m}$.
\end{proof}

\begin{exmp}
    \label{exmp:ired3}
    Consider differentiation order $m = 3$.
    Then, the proposed IRED with input $u_k = f(kT) + \eta_k$ in implicit form is given by
    \begin{equation}
        \label{eq:ired3}
        \begin{aligned}
            z_{1,k+1} &= z_{1,k} + T \lambda_1 L^{\frac{1}{4}} \spowf{u_k - z_{1,k+1}}{3}{4} + T z_{2,k+1} \\
            z_{2,k+1} &= z_{2,k} + T \lambda_2 L^{\frac{2}{4}} \spowf{u_k - z_{1,k+1}}{2}{4} + T z_{3,k+1}
            & y_{1,k} &= z_{2,k+1} + \frac{T}{2} z_{3,k+1} + \frac{T^2}{3} z_{4,k+1}, \qquad
            \\
            z_{3,k+1} &= z_{3,k} + T \lambda_3 L^{\frac{3}{4}} \spowf{u_k - z_{1,k+1}}{1}{4} + T z_{4,k+1}
            & y_{2,k} &= z_{3,k+1} + T z_{4,k+1}, \qquad
            \\
            z_{4,k+1} &\in z_{4,k} + T \lambda_4 L \spow{u_k-z_{1,k+1}}{0}
            &y_{3,k} &= z_{4,k+1},
        \end{aligned}
    \end{equation}
    yielding estimates $y_{1,k}, y_{2,k}, y_{3,k}$ for the first three derivatives $f^{(1)}(kT), f^{(2)}(kT), f^{(3)}(kT)$ of the signal $f \in \FL{3}$.
For a numerical implementation, define
\begin{subequations}
    \begin{equation}
        b_k = u_k - z_{1,k} - T z_{2,k} - T^2 z_{3,k} - T^4 z_{4,k}, \qquad
        \hat \rho_k = \begin{cases}
            0 & \text{if $|b_k| \le \lambda_4 L T^4$} \\
            \hat r_k \sign(b_k) & \text{otherwise}
        \end{cases}
    \end{equation}
    where $\hat r_k$ is an (approximate) solution of the polynomial equation
    \begin{equation}
        L T^4 r_k^4 + \lambda_1 L T^4 r_k^3 + \lambda_2 L T^4  r_k^2 + \lambda_3 L T^4 r_k + \lambda_4 L T^4 - |b_k| = 0.
    \end{equation}
    Then, the state update and differentiator outputs may be computed according to
    \begin{equation}
        \begin{aligned}
            z_{4,k+1} &= 
            \begin{cases}
                z_{4,k} + \frac{b_k}{T^3} & \text{if $|b_k| \le \lambda_4 L T^4$} \\
                z_{4,k} + \lambda_4 L T \sign(b_k) & \text{otherwise},
            \end{cases} & y_{3,k} &= z_{4,k+1} \\
            z_{3,k+1} &= z_{3,k} + T z_{4,k+1} + \lambda_3 L T^2 \hat\rho_k  &
            y_{2,k} &= z_{3,k+1} + T z_{4,k+1} \\
            z_{2,k+1} &= z_{2,k} + T z_{3,k+1} + \lambda_2 L T^3 \spow{\hat\rho_k}{2} &
            y_{1,k} &= z_{2,k+1} + \frac{T}{2} z_{3,k+1} + \frac{T^2}{3} z_{4,k+1}
            \\
            z_{1,k+1} &= z_{1,k} + T z_{2,k+1} + \lambda_1 L T^4 \spow{\hat\rho_k}{3}.
        \end{aligned}
    \end{equation}
\end{subequations}
\end{exmp}

\subsection{Exactness Properties and Absence of Discretization Chattering}

An important property of the \emph{continuous-time} robust exact differentiator \eqref{eq:diff:levant} of order $m$ is its capability to differentiate signals in $\FL{m}$ exactly in the absence of noise.
It is clear that a \emph{sample-based} differentiator cannot achieve this for $L > 0$.
To characterize the best possible approximation of exactness, Seeber and Haimovich\cite{seeber202Xoptimal},  for the case of first-order differentiation, introduced the notion of quasi-exactness:
A first-order sample-based differentiator is called quasi-exact, if its worst case differentiation error upper bound in the absence of measurement noise is minimal among all sample-based differentiators.\cite[Definition~6.4]{seeber202Xoptimal}

The following theorem gives \emph{tight differentiation error bounds} of the proposed differentiator in the absence of measurement noise, which coincide with the quasi-exactness bound for the first order case.
Moreover, it shows that the proposed differentiator achieves exact differentiation of the polynomials $\FLz{m}$ in finite time \emph{without discretization chattering} when no measurement noise is present.
As the Examples~\ref{ex:hidd} and~\ref{ex:ihdd} show, the HIDD, the I-HDD, and the I-AO-STD, in contrast, do not have this property.
The proof of the theorem is given in Section~\ref{sec:proofs:main}.
\begin{thm}
    \label{th:exactness}
    Let $m \in \NN$ and $L, T, \lambda_1, \ldots, \lambda_{m+1} \in \RR_{> 0}$.
    Let $M \in [0, L]$ and consider the sample-based implicit sliding-mode differentiator defined in \eqref{eq:diff} with input $u_k = f(kT) + \eta_k$ wherein $f \in \FLb{m}$ and $(\eta_k)$ is a real-valued sequence .
    Suppose that $K \in \NN$ exists such that for all $k \ge K$, $\eta_k = 0$ holds and the differentiator is in discrete-time sliding mode at time index $k$, i.e., such that $z_{1,k+1} = u_k = f(kT)$ holds for all $k \ge K$.
    Then, the differentiator's outputs $(\y_k) = \Diff{m}[(u_k)]$ satisfy
    \begin{equation}
        \label{eq:differror:bound}
        |y_{i,k} - f^{(i)}(kT)| \le c_{i,m+1} M T^{m-i+1}
    \end{equation}
    for all $i=1,\ldots, m$ and all $k \ge K + m+1$, with constants $c_{i,m+1}$ as defined in \eqref{eq:cij}.
    Moreover, if $f(t) = M \frac{t^{m+1}}{(m+1)!}$, then the previous statement holds with equality in \eqref{eq:differror:bound}.
\end{thm}
\begin{rem}
    This theorem shows that the proposed differentiator \eqref{eq:diff} is a proper implicit discretization in the sense of Definition~\ref{def:proper}.
Indeed, an equivalent, purely state-based implementation of the proposed differentiator which fulfills the corresponding necessary and sufficient structural conditions derived by Seeber and Koch\cite[Theorem 3.2]{seeber2023structural} may be obtained from \eqref{eq:diff:zkp1i}--\eqref{eq:diff:zkp1n} by choosing the state variables $\zeta_{1,k} = z_{1,k}$, and $\zeta_{2,k} = y_{1,k-1}, \ldots, \zeta_{m+1,k} = y_{m,k-1}$, i.e., by means of the state transform
    \begin{equation}
        \begin{bmatrix}
            \zeta_{1,k} \\
            \zeta_{2,k} \\
            \vdots \\
            \zeta_{m+1,k}
        \end{bmatrix} = \begin{bmatrix}
            1 & 0 & 0 & 0 &  \ldots &  0 \\
            0 & 1 & T c_{1,2} & T^2 c_{1,3} & \ldots & T^{m-1} c_{1,m} \\
            0 & 0 & 1 & T c_{2,3} & \ldots & T^{m-2} c_{2,m} \\
            \vdots & \ddots & \ddots &  \ddots & \vdots & \vdots \\
0 & 0 & \ddots & \ddots & 1 & T c_{m-1,m} \\
            0 & 0 & 0 & 0 & \ldots & 1
        \end{bmatrix} \begin{bmatrix}
            z_{1,k} \\
            z_{2,k} \\
            \vdots \\
            z_{m+1,k}
        \end{bmatrix}
    \end{equation}
    with constants $c_{i,j}$  defined in \eqref{eq:cij}.
    The estimates for $\dot f(kT), \ldots, f^{(m)}(kT)$ are given by the states $\zeta_{2,k+1}, \ldots, \zeta_{m+1,k+1}$ in that case.
\end{rem}

As a special case of Theorem~\ref{th:exactness}, the following corollary shows that the implicit super-twisting differentiator (I-STD)\cite{mojallizadeh2021time}, which corresponds to the first-order IRED, is quasi-exact  whenever it converges into discrete-time sliding mode.
\begin{cor}
    \label{cor:quasiexact}
    Let $L, T, \lambda_1, \lambda_2 \in \RR_{> 0}$ and consider the implicit super-twisting differentiator
    \begin{subequations}
        \label{eq:istd}
        \begin{align}
            z_{1,k+1} &= z_{1,k} + \lambda_1 L T \spowf{u_k - z_{1,k+1}}{1}{2} + T z_{2,k+1} \\
            z_{2,k+1} &\in z_{2,k} + \lambda_2 L T \spow{u_k - z_{1,k+1}}{0} &
            y_{1,k} &= z_{2,k+1}
        \end{align}
    \end{subequations}
    with initial condition $z_{1,1}, z_{2,1} \in \RR$.
    Suppose that for all $f \in \FL{1}$ there exists $K \in \NN$, depending only on $f(0), \dot f(0)$, such that the differentiator with input $u_k = f(kT)$ is in discrete-time sliding mode, i.e., $z_{1,k+1} = u_k$, for all $k \ge K$.
    Then, the differentiator \eqref{eq:istd} is quasi-exact in finite time.
\end{cor}
\begin{proof}
    Applying Theorem~\ref{th:exactness} for $m = 1$ and $M = L$ yields that $|y_{1,k} - \dot f(kT)| \le c_{1,2} L T = \frac{LT}{2}$ holds after a finite time that depends only on the initial condition, fulfilling the definition\cite[Definition~6.4]{seeber202Xoptimal} of quasi-exactness in finite time.
\end{proof}

\subsection{Stability Conditions and Robustness to Noise}

The next theorem shows stability conditions and differentiation error bounds for the proposed differentiator.
In fact, it is the first time that closed-form stability conditions and error bounds for a discrete-time implementation of the robust exact differentiator with \emph{arbitrary} differentiation order are presented.
Its proof is given in Section~\ref{sec:proofs:main}.
\begin{thm}
    \label{th:main}
    Let $m \in \NN$, $L,T \in \RR_{> 0}$, and $a_1, \ldots, a_m \in (1, 2)$.
    Define $\beta_1 = 1$, $\gamma_0 = \gamma_1 = 2$ and recursively define further constants $\beta_2, \ldots, \beta_{m+1}$, $\gamma_2, \ldots, \gamma_{m+1}$, and $\mu_1, \ldots, \mu_m$ via
    \begin{align}
        \beta_{j+1} &= \left(\beta_{j}^{j} + \frac{a_{j}}{\gamma_{j}^{j}}\right)^{\frac{1}{j}}, &
\gamma_{j+1} &= \left(\frac{2 }{2 - a_{j}}\right)^{\frac{1}{j}} \gamma_{j}, &
        \label{eq:muj}
        \mu_j = \frac{j+1}{j}\cdot \frac{\gamma_{j}^{j}}{\gamma_{j-1}^{j-1}}\cdot \frac{\beta_{j+1}}{a_{j} -1}
    \end{align}
    for $j = 1, \ldots, m$.
    Additionally, define $\lambda_0 = 1$.
Consider the sample-based implicit sliding-mode differentiator $\Diff{m}$ defined in \eqref{eq:diff} with initial values $z_{1,1}, z_{2,1}, \ldots, z_{m+1,1} \in \RR$ and suppose that its parameter $\lambda_1, \ldots, \lambda_{m+1} \in \RR_{> 0}$ satisfy $\lambda_{m+1} > 1$ and
    \begin{equation}
        \label{eq:cond}
        \frac{\lambda_{m-j+1}}{\lambda_{m-j}} > \frac{\lambda_{m-j+2}}{\lambda_{m-j+1}} \mu_{j} \qquad\text{for $j = 1, \ldots, m$}.
    \end{equation}
    Then, for every $f \in \FL{m}$ and every $(\eta_k) \in \EN$, there exists a finite integer $K \in \NN$ independent of $f$ and $(\eta_k)$ except for the initial conditions $f(0), f^{(1)}(0), \ldots, f^{(m)}(0)$ and the noise bound $N$ such that, when applying the input $u_k = f(kT) + \eta_k$ to the differentiator, its outputs $(\y_k) = \Diff{m}[(u_k)]$ satisfy the inequalities
    \begin{equation}
        |y_{i,k} - f^{(i)}(kT)| \le c_{i,m+1} L \left( T + d_i \sqrt[m+1]{\frac{N}{L}} \right)^{m-i+1} \quad
        \text{with}
        \quad
        \label{eq:di}
        d_i = \max_{p = 1, \ldots, m-i+1} \frac{\beta_p \gamma_{m+1}}{\sqrt[m+1]{2}} \sqrt[p]{\frac{\lambda_{m-p+1}c_{i,m-p+1} }{\binom{m-i+1}{p}c_{i,m+1} }}
    \end{equation}
    for all $i=1,\ldots, m$ and all $k \ge K$, with constants $c_{i,j}$ as defined in \eqref{eq:cij}.
    If, additionally, $N \le \bar N$  holds with
    \begin{equation}
        \label{eq:Nbar}
        \bar N = \min_{p=0, \ldots, m} \frac{L T^{m+1}}{2^{m} \beta_{m-p+1}^{m+1} \gamma_{m+1}^{m+1}} \left( \frac{\lambda_{m+1}-1}{\lambda_p} \right)^{\frac{m+1}{m-p+1}},
    \end{equation}
    then the differentiator moreover is in discrete-time sliding mode (i.e., $z_{1,k+1} = u_k$ holds) for all $k \ge K$.
\end{thm}

\begin{exmp}
    Consider differentiation order $m = 1$.
    Then, the theorem yields the stability conditions $\lambda_1^2 > \mu_1 \lambda_2$, $\lambda_2 > 1$ with
    \begin{equation}
        \mu_1 = 2 \gamma_1 \frac{\beta_2}{a_1 - 1} = \frac{2 a_1 + 4}{a_1 - 1}
    \end{equation}
    and $a_1 \in (1,2)$.
    Since $\inf_{a_1 \in (1,2)} \frac{2 a_1 + 4}{a_1 - 1} = 8$, this condition is satisfied for some $a_1$ whenever $\lambda_1^2 > 8 \lambda_2$, $\lambda_2 > 1$.
\end{exmp}
\begin{rem}
    For differentiation orders $m > 1$, the following tuning rule\cite[Remark 3]{seeber2023closed} may be used to obtain parameters that satisfy the theorem's conditions:
    After computing $\mu_1, \ldots, \mu_m$ as defined in the theorem, select values $\bar \mu_1, \ldots, \bar \mu_m$ each satisfying $\bar \mu_j > \mu_j$.
    Then, for any given $\lambda_{m+1} > 1$, admissible parameters $\lambda_1, \ldots, \lambda_m$ satisfying \eqref{eq:cond} may be computed according to
    \begin{equation}
        \lambda_j = \lambda_{m+1}^{\frac{j}{m+1}} \frac{\prod_{k=m-j+1}^{m} \prod_{i=1}^{k} \bar \mu_i}{\left( \prod_{k=1}^m \prod_{i=1}^k \bar \mu_i\right)^{\frac{j}{m+1}}}
    \end{equation}
    for $j = 1, \ldots, m$.
\end{rem}
\begin{rem}
    For the noise-free case $N = 0$, this theorem may be combined with Theorem~\ref{th:exactness} to obtain that also the (possibly tighter) bounds \eqref{eq:differror:bound} are established after a finite time.
    For differentiation order $m = 1$, the above example furthermore shows that $\lambda_1 > \sqrt{8 \lambda_2}$, $\lambda_2 > 1$ is sufficient for the condition of the corresponding Corollary~\ref{cor:quasiexact} to be satisfied.
\end{rem}
\begin{rem}
    From Proposition~\ref{prop:approximation}, one can see that when replacing $N$ by $N+R L T^{m+1}$ in this theorem, it applies also to the approximate numerical implementation \eqref{eq:bk}, \eqref{eq:alpha:poly:approx}, \eqref{eq:diff:approx}, \eqref{eq:diff:y} of the IRED.
    Provided that $R L T^{m+1}\le \bar N$, it can furthermore be seen that such an approximate implementation also eventually attains the discrete-time sliding mode in the noise-free case $N = 0$.
    This shows that the bound \eqref{eq:differror:bound} from Theorem~\ref{th:exactness} is valid also for the numerical implementation of the IRED provided that $R$ is sufficiently small.
\end{rem}

\section{Stability Analysis}
\label{sec:stability}

For notational convenience, $n = m+1$ denotes the system order of the differentiator throughout this section.

\subsection{Differentiation Error System}

To introduce the error system, first define an extension $\bar f : \RR \to \RR$  of $f \in \FL{m}$ to negative values of $t$ as
\begin{equation}
    \label{eq:fbar}
    \bar f(t) = \begin{cases}
        f(t) & t \ge 0 \\
        f(0) + \sum_{j=1}^{m} \frac{t^{j}}{j!} f^{(j)}(0) & t < 0.
    \end{cases}
\end{equation}
Obviously, $\bar f^{(m+1)}(t) = 0$ for $t < 0$ and hence $|\bar f^{(m+1)}| \le L$ almost everywhere on $\RR$.
For given $f$, recursively define corresponding divided differences $g^f_{i,k}$ for $i = 1, \ldots, m+2$ and all integers $k \in \ZZ$ via the relations
\begin{equation}
    \label{eq:gf}
    g^f_{1,k+1} = \bar f(kT) \qquad \text{and} \qquad
    g^f_{i+1, k+1} = \frac{g^{f}_{i,k+1} - g^{f}_{i,k}}{T} \qquad \text{for $i=1,\ldots, m+1$}.
\end{equation}
Now, define the error variables $x_{i,k} = z_{i,k} - g^{f}_{i,k}$ for $i = 1, \ldots, n = m+1$, aggregated in the vector $\x_k = \begin{bmatrix} x_{1,k} & \ldots & x_{n,k} \end{bmatrix}\TT$.
Noting that $g^{f}_{1,k+1} = f(kT) = u_k - \eta_k$ for $k \ge 0$ and from \eqref{eq:diff:zkp1i}--\eqref{eq:diff:zkp1n}, these can be seen to satisfy
\begin{subequations}
    \label{eq:diff:errorsystem}
    \begin{align}
    \label{eq:diff:errorsystem:i}
        x_{i,k+1} &= x_{i,k} - T L^{\frac{i}{n}} \lambda_i \spowf{x_{1,k+1}- \eta_k}{n-i}{n} + T x_{i+1,k+1} \qquad \text{for $i = 1, \ldots, n-1$} \\
        x_{n,k+1} &\in x_{n,k} - T L \lambda_{n} \spow{x_{1,k+1}-\eta_k}{0} - T \delta_k
    \end{align}
\end{subequations}
with $\delta_k = g^{f}_{n+1,k}$.
It is well known\cite{curtiss1962limits,livne2014proper} that $|\bar f^{(m+1)}| \le L$ holding almost everywhere implies $|\delta_k| = |g^{f}_{n+1,k}| \le L$ for all $k \in \NN$.

The following lemma shows a forward invariant set in which the system operates in sliding mode for sufficiently small measurement noise.
It and all following lemmata are proven in the appendix.
\begin{lem}
    \label{lem:invariant}
    Let $n \in \NN$, $n \ge 2$ and $L, T, \lambda_1, \ldots, \lambda_n \in \RR_{> 0}$.
    Define the set
    \begin{equation}
        \label{eq:invariant}
        \Omega = \{ \x \in \RR^n : 2^{n-i+1} T^{i-1} \abs{x_i} \le L T^n (\lambda_n - 1) \text{ for $i = 1, \ldots, n$} \},
    \end{equation}
    and let $N \in [0, 2^{-n} L T^n (\lambda_n-1))$.
    Consider solutions of system \eqref{eq:diff:errorsystem} with $|\eta_k| \le N$ and $|\delta_k| \le L$ for all $k\ge 1$.
    Then, for all $\bar K \in \NN$, $\x_{\bar K} \in \Omega$ implies that $\x_{\bar K+1} \in \Omega$ and that $x_{1,\bar K+1} = \eta_{\bar K}$.
\end{lem}
The next lemma gives bounds on the states $x_{i,k}$ that eventually are established in discrete-time sliding-mode. 
\begin{lem}
    \label{lem:deadbeat}
    Let $n \in \NN$, $n \ge 2$ and $N, L, T, \lambda_1, \ldots, \lambda_n \in \RR_{> 0}$.
    Consider system \eqref{eq:diff:errorsystem} and suppose that $x_{1,k+1} = \eta_k \in [-N,N]$ holds for all $k \ge K$ and some $K \in \NN$.
    Then, $|x_{i,k}| \le (2/T)^{i-1} N$ holds for all $k \ge K + i$ and $i = 1, \ldots, n$.
\end{lem}
The following lemma establishes tight bounds on the differentiation errors, once the error states $x_{i,k}$ vanish.
\begin{lem}
    \label{lem:errorbound}
    Let $m \in \NN$, $M \in \RR_{\ge 0}$, $T \in \RR_{> 0}$, and define the constants $c_{i,j}$ as in \eqref{eq:cij}.
    For $f \in \FLb{m}$, consider the divided differences $g^{f}_{i,k}$ as defined in \eqref{eq:gf}.
    Then,
    \begin{equation}
        \label{eq:lem:errorbound}
        |f^{(i)}(kT) - \sum_{j=i}^{m} T^{j-i} c_{i,j} g^{f}_{j+1,k+1}| \le c_{i,m+1} M T^{m+1-i}
    \end{equation}
    holds for all $k \in \NN_0$.
    Moreover, if $f(t) = \frac{M t^{m+1}}{(m+1)!}$, then the previous statement holds with equality in \eqref{eq:lem:errorbound} for all $k \ge m$.
\end{lem}

The following proposition now shows how the initial values of the signal $f$ and of the error system \eqref{eq:diff:errorsystem} are related.
\begin{prop}
    \label{prop:initial}
    Let $m \in \NN$, $L,T \in \RR_{> 0}$ and $f \in \FL{m}$, and consider divided differences $g^{f}_{i,k}$ as defined in \eqref{eq:gf} and the constants $c_{i,j}$ as defined in \eqref{eq:cij}.
    Then, $g_{1,1}, \ldots, g_{n,1}$ satisfy
    \begin{equation}
        \label{eq:initial}
        \begin{bmatrix}
T g^f_{2,1} \\
            T^2 g^f_{3,1} \\
            \vdots \\
            T^{n-1} g^f_{n,1}
        \end{bmatrix} = \begin{bmatrix}
c_{1,1} & c_{1,2} & \ldots & c_{1,n-1} \\
            0 & c_{2,2} & \ldots & c_{2,n-1} \\
            \vdots & \ddots & \ddots & \vdots \\
            0 & 0 & \ldots & c_{n-1,n-1}
        \end{bmatrix}^{-1} \begin{bmatrix}
T f^{(1)}(0) \\
            T^2 f^{(2)}(0) \\
            \vdots \\
            T^{n-1} f^{(n-1)}(0)
        \end{bmatrix}
    \end{equation}
and $g_{1,1} = f(0)$.
\end{prop}
\begin{proof}
    The fact that $g_{1,1} = f(0)$ is clear from the definition \eqref{eq:gf}.
    Consider the function $h \in \FLz{m}$ defined as
    \begin{equation}
        h(t) = f(0) + \sum_{j=1}^{m} \frac{t^{j}}{j!} f^{(j)}(0),
    \end{equation}
    which coincides with $\bar f(t)$ defined in \eqref{eq:fbar} for $t \le 0$.
    Hence, $g^{f}_{i,1} = g^{h}_{i,1}$ holds for all $i$, because evaluation of those divided differences only involves values of $\bar f(t)$ for $t \le 0$.
    Applying Lemma~\ref{lem:errorbound} to $h \in \FLz{m}$ with $M = k = 0$ yields
    \begin{equation}
        T^{i} f^{(i)}(0) = T^{i} h^{(i)}(0) = T^{i} \sum_{j=i}^{m} T^{j-i} c_{i,j} g^{h}_{j+1,1} = \sum_{j=i}^{m} c_{i,j} T^{j} g^{f}_{j+1,1}
    \end{equation}
    from which \eqref{eq:initial} is obtained after solving for $T^j g^f_{j+1,1}$.
\end{proof}

The next proposition shows differentiation error bounds, provided that the error states $x_{i,k}$ are ultimately bounded.
\begin{prop}
    \label{prop:errorbound}
    Let $m \in \NN$, $n = m+1$ and $L \in \RR_{> 0}$, $N \in \RR_{\ge 0}$, $M \in [0, L]$.
    Consider the sample-based differentiator $\Diff{m}$ with sampling time $T \in \RR_{> 0}$ and parameters $\lambda_1, \ldots, \lambda_n \in \RR_{> 0}$.
    For given $f \in \FLb{m}$ and $(\eta_k) \in \EN$, consider the error system \eqref{eq:diff:errorsystem} and suppose that there exists a finite integer $K \in \NN$ and constants $\psi_0, \ldots, \psi_{m} \in \RR_{> 0}$ such that
    \begin{equation}
        \label{eq:xbound:assumption}
        |x_{i+1,k}| \le \psi_{i} L^{\frac{i}{n}} N^{\frac{n-i}{n}}
    \end{equation}
    holds for $i = 0, \ldots, m$ and all $k \ge K$.
    Then, the output $(\y_k) = \Diff{m}[ (u_k) ]$ of the differentiator with input $u_k = f(kT) + \eta_k$ satisfies
    \begin{equation}
        \label{eq:differror:psiLbar}
        |y_{i,k} - f^{(i)}(kT)| \le c_{i,n} M T^{n-i} + \sum_{j=i}^{m} c_{i,j} \psi_{j} L^{\frac{j}{n}} N^{\frac{n-j}{n}} T^{j-i}
    \end{equation}
for $i = 1, \ldots, m$ and all $k \ge K$.
    Moreover, if $N = 0$, $K \ge m$, and $f(t) = \frac{M t^{n}}{n!}$, then the previous statement holds with equality in \eqref{eq:differror:psiLbar}.
\end{prop}
\begin{proof}
    From Lemma~\ref{lem:errorbound}, relation \eqref{eq:diff:y}, and $z_{i,k} = x_{i,k} + g^{f}_{i,k}$, obtain
    \begin{align}
        | f^{(i)}(kT) - y_{i,k} | &= \bigl|f^{(i)}(kT) - \sum_{j = i}^{m} T^{j-i} c_{i,j} (x_{j+1,k+1} + g^{f}_{j+1,k+1})\bigr| \le |f^{(i)}(kT) - \sum_{j = i}^{m} T^{j-i} c_{i,j} g^{f}_{j+1,k+1}| + \sum_{j=i}^{m} T^{j-i} c_{i,j} |x_{j+1,k+1}| \nonumber \\
        &\le c_{i,n} M T^{n-i} + \sum_{j=i}^{m} c_{i,j} \psi_{j} L^{\frac{j}{n}} N^{\frac{n-j}{n}} T^{j-i},
    \end{align}
    proving \eqref{eq:differror:psiLbar}.
    For $N = 0$, equality for $f(t) = \frac{M t^n}{n!}$ is also obtained from that lemma.
\end{proof}

\subsection{Transformed Error System}

In order to construct a Lyapunov function for the error system, a state transformation similar to the one proposed by Cruz-Zavala and Moreno\cite{cruz2019levant} is first applied.
To that end, define the parameters
$
    \kappa_j = \lambda_{n-j+1}/\lambda_{n-j}
$
for $j = 0, \ldots, n$ with the abbreviations $\lambda_0 = \lambda_{n+1} = 1$, i.e.,
\begin{equation}
    \kappa_0 = \frac{1}{\lambda_n}, \quad
    \kappa_1 = \frac{\lambda_n}{\lambda_{n-1}}, \quad
    \ldots, \quad
    \kappa_{n-1} = \frac{\lambda_2}{\lambda_{1}}, \quad
    \kappa_{n} = \frac{\lambda_1}{1},
\end{equation}
and introduce the state transform
$
    \xi_j = x_{n-j+1}/(\lambda_{n-j} L)
$
for $j = 1, \ldots, n$, i.e.,
\begin{equation}
    \xi_1 = \frac{x_n}{\lambda_{n-1} L}, \quad
	\xi_2 = \frac{x_{n-1}}{\lambda_{n-2} L}, \quad
    \ldots, \quad
    \xi_{n-1} = \frac{x_2}{\lambda_1 L}, \quad
    \xi_n = \frac{x_1}{L},
\end{equation}
which reverses and scales the state variables of \eqref{eq:diff:errorsystem}, yielding the system
\begin{equation}
    \label{eq:diff:errorsystem:transformed}
    \begin{aligned}
        \xi_{1,k+1} &\in \xi_{1,k} - T \kappa_1( \spow{\xi_{n,k+1} - \tilde\eta_k}{0} + \kappa_0 \tilde\delta_k ) \\ 
        \xi_{2,k+1} &= \xi_{2,k} - T \kappa_2( \spowf{\xi_{n,k+1} - \tilde\eta_k}{1}{n} - \xi_{1,k+1} ) \\ 
&\vdots \\
        \xi_{n-1,k+1} &= \xi_{n-1,k} - T \kappa_{n-1}( \spowf{\xi_{n,k+1} - \tilde\eta_k}{n-2}{n} - \xi_{n-2,k+1} ) \\
        \xi_{n,k+1} &= \xi_{n,k} - T \kappa_n( \spowf{\xi_{n,k+1} - \tilde\eta_k}{n-1}{n} - \xi_{n-1,k+1} )
    \end{aligned}
\end{equation}
with $\tilde\eta_k = \frac{\eta_k}{L}$ and $\tilde \delta_k = \frac{\delta_k}{L}$.
Using the convenient abbreviations $\eta_{n+1,k+1} = -\tilde\eta_k$ and $\xi_{0,k+1} = -\kappa_0 \tilde\delta_k$, this system may be rewritten as
\begin{equation}
    \label{eq:diff:errorsystem:transformed2}
    \begin{aligned}
        \xi_{1,k+1} &= \xi_{1,k} - T \kappa_1 \eta_{1,k+1} \qquad & \eta_{1,k+1} &\in \spow{\xi_{1,k+1} + \eta_{2,k+1}}{0} -  \xi_{0,k+1} \\ 
        \xi_{2,k+1} &= \xi_{2,k} - T \kappa_2 \eta_{2,k+1} & \eta_{2,k+1} &=  \spowf{\xi_{2,k+1} + \eta_{3,k+1}}{1}{2} - \xi_{1,k+1} \\ 
&\vdots \\
\xi_{n,k+1} &= \xi_{n,k} - T \kappa_n \eta_{n,k+1} & \eta_{n,k+1} &= \spowf{\xi_{n,k+1} + \eta_{n+1,k+1}}{n-1}{n} - \xi_{n-1,k+1}.
    \end{aligned}
\end{equation}
This recursive form of the system admits the recursive construction of a Lyapunov function\cite{seeber2023closed}, which is studied in the following.

\subsection{Lyapunov Function}

In the following, a discrete-time Lyapunov function for the error system \eqref{eq:diff:errorsystem} is proposed.
The same Lyapunov function has already been used to analyze the continuous-time robust exact differentiator in a conference paper\cite{seeber2023closed} subject to some additional technical assumptions.
However, it does not have convex sublevel sets, and hence existing approaches\cite{efimov2017realization,brogliato2020implicit} for transferring this analysis to the proposed implicitly discretized differentiator are not applicable.
Hence, the following analysis is performed purely in discrete time without relying on continuous-time results.

Introduce the state vector $\vxi = [ \xi_1 \quad \ldots \quad \xi_n ]\TT$, recursively define the positive semidefinite functions $V_j : \RR^n \to \RR_{\ge 0}$ as
\begin{subequations}
    \label{eq:V}
    \begin{align}
        V_1(\vxi) &= \abs{\xi_1}, \\
        V_{j}(\vxi) &= \max\left\{ V_{j-1}(\vxi), \alpha_{j}^{-\frac{1}{j-1}} \abs{\spowf{\xi_{j}}{j-1}{j} - \xi_{j-1}}^{\frac{1}{j-1}}  \right\} \qquad \text{for $j=2, \ldots, n$}
    \end{align}
\end{subequations}
with positive parameters $\alpha_2, \ldots, \alpha_{n}$.
Note that $V_1(\vxi) \le V_2(\vxi) \le \ldots \le V_n(\vxi)$ holds by construction and consider $V = V_n$ as a Lyapunov function candidate.
The following lemma introduces conditions on the free parameters and shows bounds on the state variables in terms of these functions $V_1, \ldots, V_n$.

\begin{lem}
    \label{lem:Vbound}
    Let $n \in \NN$ and suppose that $\alpha_2, \ldots, \alpha_{n}$, $\beta_1, \ldots, \beta_n$, $\gamma_1, \ldots, \gamma_n \in \RR_{> 0}$ satisfy $\beta_1 = 1$, $\gamma_1 = 2$, and
    \begin{align}
        \label{eq:recursion}
        \alpha_{j+1}\gamma_j^j &\in ( 1, 2 ), &
\beta_{j+1} &= (\beta_{j}^{j} + \alpha_{j+1})^{\frac{1}{j}}, &
\gamma_{j+1} &= \left(\frac{2 \gamma_{j}^{j}}{2 - \alpha_{j+1} \gamma_{j}^{j}}\right)^{\frac{1}{j}},
    \end{align}
    for $j = 1, \ldots, n-1$.
    Consider functions $V_1, \ldots, V_n$ defined in \eqref{eq:V}.
    Then, the inequality
$
\abs{\xi_j} \le \beta_j^j V_j(\vxi)^{j}
$
holds for all $\vxi \in \RR^n$ and all $j = 1, \ldots, n$; in particular, $V_n$ is positive definite and radially unbounded.
\end{lem}

The next lemma, which is proven by induction over the system order $n$, establishes that the recursively constructed, positive definite function $V = V_n$ is indeed a Lyapunov function for system \eqref{eq:diff:errorsystem:transformed2}.
\begin{lem}
    \label{lem:main}
Let $n \in \NN$ and let $\alpha_2, \ldots, \alpha_{n}$, $\beta_1, \ldots, \beta_n$, $\gamma_1, \ldots, \gamma_n \in \RR_{> 0}$ satisfy the conditions of Lemma~\ref{lem:Vbound}.
    Suppose that $\kappa_0, \kappa_1, \ldots, \kappa_n \in \RR_{> 0}$ satisfy $\kappa_0 \in (0,1)$, $\kappa_1 > 0$, and
    \begin{equation}
\label{eq:cond:aux}
\frac{\kappa_p}{p}\frac{\alpha_{p} \gamma_p^{p-1} - 2}{\beta_p} > \frac{\kappa_{p-1}}{p-1} \biggl( \alpha_{p-1} \gamma_p^{p-1} + 2^{\frac{1}{p-1}} \gamma_p\left( \alpha_{p} \gamma_p^{p-1} + 2 \right)^{\frac{p-2}{p-1}}\biggr)
    \end{equation}
    for $p = 2, \ldots, n$ with $\alpha_1 = 0$.
    Consider system \eqref{eq:diff:errorsystem:transformed2} and the functions $V_1, \ldots, V_n$ defined in \eqref{eq:V}.
    Then, for all $N \in \RR_{> 0}$ there exist positive constants $\epsilon_1, \ldots, \epsilon_n \in \RR_{> 0}$ such that for all $k \ge 1$ and all $j =1, \ldots, n$ the three inequalities $V_j(\vxi_{k+1}) >  \gamma_j \sqrt[j]{ N/(2L) }$, $|\eta_{j+1,k+1}|\le \frac{N}{L}$, and $|\xi_{0,k+1}| \le \kappa_0$ imply 
$
        V_j(\vxi_{k+1}) \le V_j(\vxi_{k}) - \epsilon_j.
$
\end{lem}
By setting $j = n$ in this lemma, the following statement is obtained for the Lyapunov function candidate $V = V_n$ subject to the lemma's conditions:
for all $N > 0$ there exists a positive constant $\epsilon_n$ (dependent on $N$ and on the differentiator parameters) such that $V(\vxi_{k+1}) > \gamma_n \sqrt[n]{N/(2L)}$ implies $V(\vxi_{k+1} \le V(\vxi_k) - \epsilon_n$, provided that $|L \eta_{n+1,k+1}| = |\eta_k| \le N$ and $|L \xi_{0,k+1}/\kappa_0| = |\delta_k| \le L$.
    By using this statement, ultimate error bounds for the error system are shown next.

\subsection{Ultimate Error Bounds}

\begin{prop}
    \label{prop:ultimate}
    Let $n \in \NN$, $L, T \in \RR_{> 0}$, $N \in \RR_{\ge 0}$.
    Suppose that the constants $\beta_1, \ldots, \beta_n$, $\gamma_1, \ldots, \gamma_n \in \RR_{> 0}$ and the differentiator parameters $\lambda_1, \ldots, \lambda_n \in \RR_{> 0}$ satisfy the conditions of Lemma~\ref{lem:main}, and consider the differentiation error system \eqref{eq:diff:errorsystem}.
    Then, for all initial values $x_{1,1}, \ldots, x_{n,1} \in \RR$ there exists a finite integer $K \in \NN$ such that every solution of \eqref{eq:diff:errorsystem} with $|\delta_k| \le L$ and $|\eta_k| \le N$ satisfies
    \begin{equation}
        \label{eq:xbound}
    |x_{i+1,k}| \le \psi_{i} N^{\frac{n-i}{n}} L^{\frac{i}{n}} \qquad \text{with} \qquad \psi_i = \lambda_i \left( \frac{\beta_{n-i} \gamma_{n}}{\sqrt[n]{2}} \right)^{n-i}
    \end{equation}
    for all $i = 0, \ldots, n-1$ and all $k \ge K$.
\end{prop}
\begin{proof}
    Consider first the case $N > 0$.
Applying Lemma~\ref{lem:main} with $j = n$ guarantees existence of a positive constant $\epsilon_n$ such that $V_n(\vxi_{k+1}) \le V_n(\vxi_k) - \epsilon_n$ holds whenever $V_n(\vxi_{k+1}) > \gamma_n \sqrt[n]{N/(2L)}$.
    Hence, there exists $K \in \NN$ depending on the initial value $V_n(\vxi_{1})$ such that $V_n(\vxi_{k}) \le \gamma_n \sqrt[n]{N/(2L)}$ holds for all $k \ge K$.
    With Lemma~\ref{lem:Vbound}, this implies
    \begin{equation}
        |x_{i+1,k}| = \lambda_i L |\xi_{n-i,k} | \le \lambda_i L \beta_{n-i}^{n-i} V_{n-i}(\vxi_k)^{n-i} \le \lambda_i L \beta_{n-i}^{n-i} V_n(\vxi_k)^{n-i} \le L \lambda_i \beta_{n-i}^{n-i} \gamma_n^{n-i} \left(\frac{N}{2L}\right)^{\frac{n-i}{n}} = \psi_{i} N^{\frac{n-i}{n}} L^{\frac{i}{n}}
    \end{equation}
    for all $k \ge K$.

    For $N = 0$, only $\lim_{k \to \infty} V_n(\vxi_k) = 0$ may be concluded from the previous considerations.
    Due to positive definiteness of $V_n$, this implies $\lim_{k \to \infty} \vxi_k = \lim_{k \to \infty} \x_k = \bm{0}$.
    Since the invariant set $\Omega$ from Lemma~\ref{lem:invariant} contains a neighborhood of the origin, there then exists $\bar K \in \NN$ such that $\x_k \in \Omega$ and $x_{1,k+1} = \eta_k = 0$ for all $k \ge \bar K$.
    Then, Lemma~\ref{lem:deadbeat} implies $\x_{k} = \bm{0}$ for all $k \ge K = \bar K + n$, concluding the proof.
\end{proof}

\subsection{Proofs of the Main Theorems}
\label{sec:proofs:main}

With the obtained results, the main theorems may now be proven.
\begin{proof}[Proof of Theorem~\ref{th:exactness}]
    According to Definition~\ref{def:slidingmode}, the differentiator is in discrete-time sliding mode iff $\spow{z_{1,k+1} - u_k}{0}$ contains more than one element, i.e., if and only if $z_{1,k+1} - u_k = 0$, or equivalently $x_{1,k+1} - \eta_k = 0$ for the error system.
    Using also $\eta_k = 0$, Lemma~\ref{lem:deadbeat} allows to conclude that $\x_k = \bm{0}$ for all $k \ge K + m + 1$.
    The claim then follows from Proposition~\ref{prop:errorbound} with $N = 0$.
\end{proof}

\begin{proof}[Proof of Theorem~\ref{th:main}]
    First note that with $\alpha_1 = 0$ and
    \begin{equation}
        \alpha_{j} = \frac{a_{j-1}}{\gamma_{j-1}^{j-1}} \qquad \text{for $j = 2, \ldots, n$}
    \end{equation}
    the constants $\alpha_j, \beta_j, \gamma_j$ satisfy the conditions of Lemma~\ref{lem:Vbound}.
    Moreover, \eqref{eq:cond} is equivalent to condition \eqref{eq:cond:aux} of Lemma~\ref{lem:main}.
    To see this, rewrite \eqref{eq:cond:aux} as $\kappa_{p} \ge \frac{p}{p-1} \beta_p \tilde \mu_{p-1} \kappa_{p-1}$ with
    \begin{equation}
        \tilde \mu_{p-1} = \frac{\alpha_{p-1} \gamma_p^{p-1} + 2^{\frac{1}{p-1}} \gamma_p\left( \alpha_{p} \gamma_p^{p-1} + 2 \right)^{\frac{p-2}{p-1}}}{\alpha_{p} \gamma_p^{p-1} - 2}
    \end{equation}
    By using the \eqref{eq:recursion} for $j = p-1$,  the numerator may be simplified as
    \begin{equation}
        \tilde \mu_{p-1} = \frac{\alpha_{p-1} \gamma_p^{p-1} + 2^{\frac{1}{p-1}} \gamma_p^{p-1} \left( \alpha_{p}  + 2 \gamma_p^{-(p-1)}\right)^{\frac{p-2}{p-1}}}{\alpha_{p} \gamma_p^{p-1} - 2} = 
        \frac{\alpha_{p-1}  + 2^{\frac{1}{p-1}}  \left( \alpha_{p}  + \frac{2 - \alpha_{p} \gamma_{p-1}^{p-1}}{\gamma_{p-1}^{p-1}} \right)^{\frac{p-2}{p-1}}}{\alpha_{p}  - 2\gamma_p^{-(p-1)}} =  \frac{\alpha_{p-1} + 2 \gamma_{p-1}^{-(p-2)}}{\alpha_{p} - 2 \gamma_{p}^{-(p-1)}}.
    \end{equation}
    In case $p > 2$, substituting the recursion again for $j = p-2$ and $j = p-1$ yields,
    \begin{equation}
        \tilde \mu_{p-1} = \frac{\alpha_{p-1} + \frac{2 - \alpha_{p-1}\gamma_{p-2}^{p-2}}{\gamma_{p-2}^{p-2}}}{\alpha_{p} - \frac{2 - \alpha_{p}\gamma_{p-1}^{p-1}}{\gamma_{p-1}^{p-1}}} = \frac{\gamma_{p-2}^{-(p-2)}}{\alpha_{p} - \gamma_{p-1}^{-(p-1)}} = \frac{\gamma_{p-1}^{p-1}}{\gamma_{p-2}^{p-2}} \frac{1}{\alpha_{p} \gamma_{p-1}^{p-1} - 1} = \frac{\gamma_{p-1}^{p-1}}{\gamma_{p-2}^{p-2}} \frac{1}{a_{p-1} - 1}.
    \end{equation}
    With $\mu_{p-1} = \frac{p \beta_p}{p-1} \tilde \mu_{p-1}$ and setting $p = j+1$, \eqref{eq:cond} is then obtained with $\mu_j$ as in \eqref{eq:muj}.

    Due to Proposition~\ref{prop:initial} and $x_{i,k} = z_{i,k} -g^f_{i,k}$, the initial values $x_{1,1}, \ldots, x_{n,1}$ depend only on the initial conditions $f(0), f^{(1)}(0), \ldots, f^{(n-1)}(0)$ of the signal $f$, and $f \in \FL{m}$ guarantees that $\delta_k = g^f_{n+1,k}$ satisfies $|\delta_k| \le L$ for all $k \in \NN$.
    Hence, Proposition~\ref{prop:ultimate} guarantees existence of $K \in \NN$ such that \eqref{eq:xbound:assumption} holds for all $k \ge K$ with $\psi_i$ in \eqref{eq:xbound}.
    Moreover, it is readily verified that $\psi_{n-p} c_{i,n-p} \le \binom{n-i}{p} c_{i,n}d_i^{p}$ holds for $i = 1, \ldots, n-1$ and $p = 1, \ldots, n-i$.
    Consequently, Proposition~\ref{prop:errorbound} with $M = L$ yields
    \begin{align}
        |y_{i,k} - f^{(i)}(kT)| &\le c_{i,n} L T^{n-i} + \sum_{j=i}^{m} c_{i,j} \psi_{j} L^{\frac{j}{n}} N^{\frac{n-j}{n}} T^{j-i} = c_{i,n} L T^{n-i} + \sum_{p=1}^{n-i} c_{i,n-p} \psi_{n-p} L^{\frac{n-p}{n}} N^{\frac{p}{n}} T^{n-i-p} \nonumber \\
        &\le c_{i,n} L T^{n-i} + \sum_{p=1}^{n-i} \binom{n-i}{p} c_{i,n} d_i^p L^{\frac{n-p}{n}} N^{\frac{p}{n}} T^{n-i-p} 
        = c_{i,n} L \sum_{p=0}^{n-i} \binom{n-i}{p} d_i^p L^{-\frac{p}{n}} N^{\frac{p}{n}} T^{n-i-p} \nonumber \\
        &= c_{i,n} L \left( T + d_i \sqrt[n]{\frac{N}{L}} \right)^{n-i},
    \end{align}
    for all $k \ge K$, proving the first claim.

    To show that the differentiator eventually is in discrete-time sliding mode if $N \le \bar N$, it will first be shown that $\x_{k} \in \Omega$ holds for all $k \ge K$, with $\Omega$ as in \eqref{eq:invariant} and $K$ as above.
    To see this, note that the above application of Proposition~\ref{prop:ultimate} ensures that
    \begin{equation}
        |x_{p+1,k}| \le \psi_{p} N^{\frac{n-p}{n}} L^{\frac{p}{n}} \le \psi_{p} \bar N^{\frac{n-p}{n}} L^{\frac{p}{n}} \le \psi_p \frac{2^{p-n} L T^{n-p}}{2^{-\frac{n-p}{n}} \beta_{n-p}^{n-p} \gamma_n^{n-p}} \left( \frac{\lambda_n-1}{\lambda_p} \right) = 2^{p-n} L T^{n-p} (\lambda_n - 1)
    \end{equation}
    holds for $p = 0, \ldots, n-1$, i.e., $\x_k \in \Omega$, for $k \ge K$.
    Since $N \le \bar N$ and $N < 2^{-n} L T^{n} (\lambda_n-1)$ according to \eqref{eq:Nbar} for $p = 0$ due to $\beta_n > 1$, $\gamma_n > 2$, Lemma~\ref{lem:invariant} may be used to obtain that $x_{1,k+1} = \eta_k$, and thus $z_{1,k+1} = x_{1,k+1} + f(kT) = u_k$ holds for all $k \ge K$, i.e., that the differentiator is in discrete-time sliding mode according to Definition~\ref{def:slidingmode}, because the set $\spow{z_{1,k+1} - u_k}{0} = \spow{0}{0} = [-1,1]$ contains more than one element.
\end{proof}

\section{Simulation Results}
\label{sec:simulation}

In the following, the proposed implicit robust exact differentiator (IRED) is compared to two other state-of-the-art differentiators: the homogeneous implicit discrete-time differentiator (HIDD) proposed by Carvajal-Rubio et al.\cite{carvajal2021implicit} and the implicit homogeneous discrete-time differentiator (I-HDD) introduced by Mojallizadeh et al.\cite{mojallizadeh2021time}
For comparison purposes, the simulation scenario is chosen as in the former paper.\cite[Section~5.2.1 \& Figure~4]{carvajal2021implicit}
Specifically, the signal to be differentiated is chosen as $f(t) = \sin t - \cos \frac{t}{2}$, the sampling time is $T = 0.1$, and differentiator order $m = 3$ and parameters $L = 2$, $\lambda_1 = 3, \lambda_2 = 4.16, \lambda_3 = 3.06, \lambda_4 = 1.1$ are used.
For an accurate numerical implementation, the tolerance parameter $R = 5\cdot 10^{-7}$ is selected for approximately solving polynomial equations by means of Newton's method as discussed in Section~\ref{sec:implementation}.

Figure~\ref{fig:noise-free} shows simulation results in the absence of measurement noise.
One can see that, after a transient convergence phase that is similar for all three approaches, the proposed IRED exhibits the best steady-state accuracy, i.e., performs best in terms of approximating the exactness of the continuous-time RED, in accordance with Theorem~\ref{th:exactness} and Corollary~\ref{cor:quasiexact}.
The HIDD, in contrast, exhibits an additional discretization chattering---as expected from Example~\ref{ex:hidd}---leading to larger and faster oscillations of the differentiation error.
The I-HDD does not exhibit such chattering but, as explained in Example~\ref{ex:ihdd}, features a bias that leads to  significantly larger differentiation error  amplitudes, especially for the first and second derivative.
Figure~\ref{fig:noisy}, finally, shows simulation results with the same setup but with additive noise samples $\eta_k$ chosen as independent and evenly distributed random numbers from the interval $[-N,N]$ with $N = 0.1$.
In this case, one can see that the differentiation errors are predominantly determined by the noise and are of comparable magnitude for all three differentiators.

\begin{figure}[tbp]
    \centering
    \includegraphics{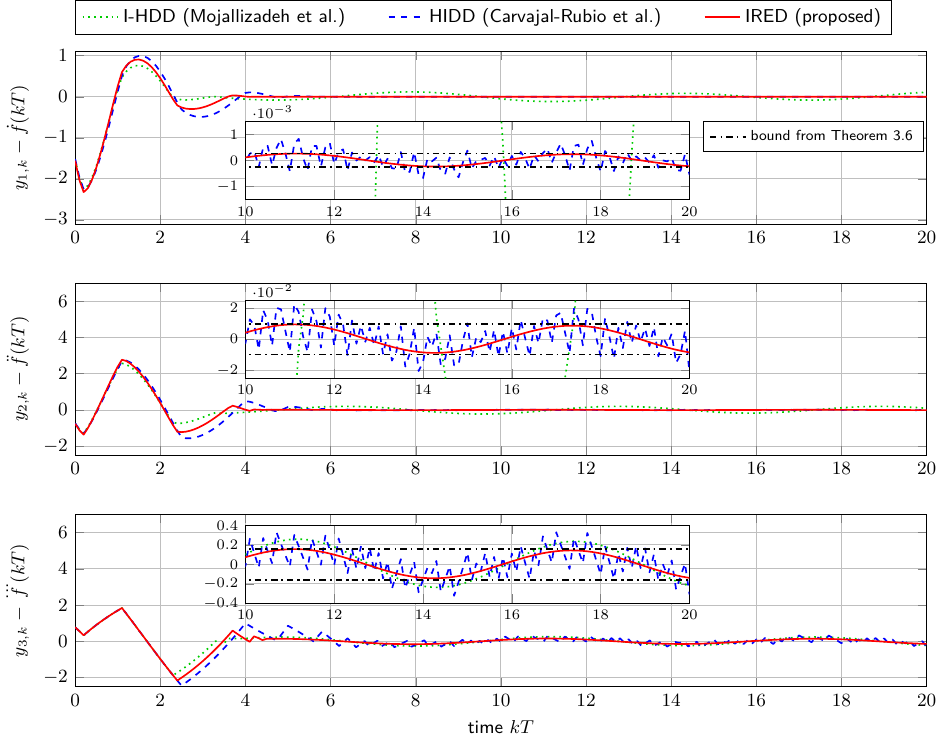}
    \caption{Comparison of I-HDD (green, dotted), HIDD (blue, dashed), and proposed ISHD (red) without measurement noise along with the bound obtained from Theorem~\ref{th:exactness} for $M = \frac{17}{16}$; for comparison with further approaches, see also the simulation comparison by Carvajal-Rubio et al.\cite[Figure 4]{carvajal2021implicit} with the same parameter setting}
    \label{fig:noise-free}
\end{figure}

\begin{figure}[tbp]
    \centering
    \includegraphics{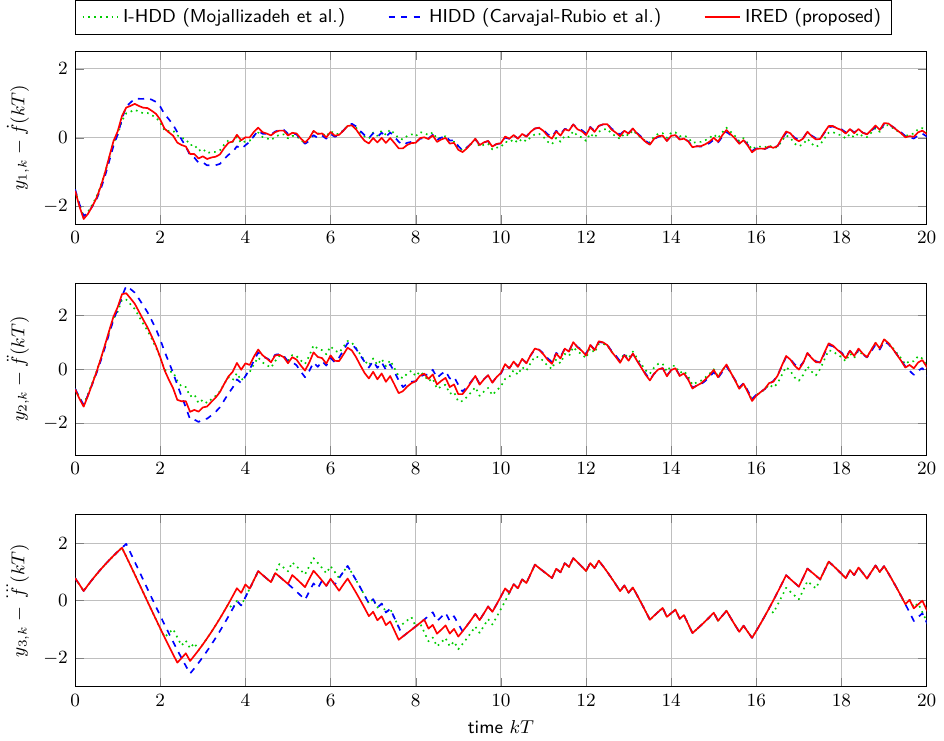}
    \caption{Comparison of I-HDD (green, dotted), HIDD (blue, dashed), and proposed ISHD (red) with simulation setup as in Figure~\ref{fig:noise-free} but with additional, independently uniformly distributed measurement noise $\eta_k \in [-N, N]$ with bound $N = 0.1$}
    \label{fig:noisy}
\end{figure}

\section{Conclusion}
\label{sec:conclusion}

A new implicit discretization of the arbitrary order robust exact differentiator, the implicit robust exact differentiator (IRED), was proposed.
In contrast to existing implicit discretizations, the proposed approach exhibits neither discretization chattering nor bias errors.
For practical use of the differentiator, an approximate numerical implementation of the differentiator was suggested and the influence of approximation errors was formally analyzed.
Compared to existing approaches, the accuracy of the proposed differentiator was shown, both formally and by simulations, to be superior at approximating exactness in the noise-free case.
In simulations, it was furthermore seen to exhibit comparable performance in presence of measurement noise.
Viability and robustness of the approach for arbitrary differentiation orders was furthermore shown by deriving closed-form conditions for finite-time stability of the differentiation error, along with error bounds in presence of measurement noise.

\section*{Acknowledgments}

The author would like to thank his colleague Stefan Koch, Graz University of Technology, for several valuable discussions.
The financial support by the Christian Doppler Research Association, the Austrian Federal Ministry for Digital and Economic Affairs and the National Foundation for Research, Technology and Development is gratefully acknowledged.

\appendix
\section{Proofs}
\label{app:proofs}

\subsection{Auxiliary Lemmata}

A few auxiliary lemmata are first stated that are required in some of the proofs.
\begin{lem}
    \label{lem:auxlem}
    For all $x,\eta \in \RR$, the inequality
    \begin{equation}
        \abs{\spowf{x + \eta}{r-1}{r} - \spowf{x}{r-1}{r}} \le 2^{\frac{1}{r}} |\eta|^{\frac{r-1}{r}}.
    \end{equation}
    holds for all positive integers $r$.
\end{lem}
\begin{proof}
    It is well-known\cite[Lemma 7]{cruz2018strict} that $|\spow{x_1}{p}+\spow{x_2}{p}|^{1/p} \le 2^{\frac{1}{p} - \frac{1}{q}} | \spow{x_1}{q} + \spow{x_2}{q} |^{1/q}$ holds for all $x_1, x_2 \in \RR$ and all $q \ge p > 0$.
    Set $p = \frac{r-1}{r}$, $q = 1$, $x_1 = x + \eta$, $x_2 = -x$ to obtain the claimed inequality.
\end{proof}

\begin{lem}
    \label{lem:auxlem2}
    Let $M \in \RR_{> 0}$, $x \in [-M, M]$ and $y \in [0, M-x]$.
    Then,
    \begin{equation}
        \spowf{x+y}{r-1}{r} \ge \spowf{x}{r-1}{r} + \frac{r-1}{r} M^{-\frac{1}{r}} y
    \end{equation}
    holds for all positive integers $r$.
\end{lem}
\begin{proof}
    For fixed $x \in [-M,M]$ consider the function $h : [0, M-x] \to \RR$ defined as
    \begin{equation}
        h(y) = \spowf{x+y}{m-1}{m} - \frac{m-1}{m} M^{-\frac{1}{m}} y.
    \end{equation}
    Since $|x+y| \le M$ for all $y \in [0, M-x]$, its derivative satisfies
    \begin{equation}
        \deriv{h}{y} = \frac{r-1}{r} |x+y|^{-\frac{1}{r}} - \frac{r-1}{r} M^{-\frac{1}{r}} \ge 0
    \end{equation}
    and the claim follows from $h(y) \ge h(0)$.
\end{proof}

\begin{lem}
    \label{lem:auxlem4}
    Let $x \in \RR_{> 0}$, $y \in \RR$ and suppose that $x + y \ge 0$.
    Then,
    \begin{equation}
        (x+y)^{\frac{1}{r-1}} \le x^{\frac{1}{r-1}} + \frac{1}{r-1} x^{-\frac{r-2}{r-1}} y
    \end{equation}
    holds for all integers $r \ge 2$.
\end{lem}
\begin{proof}
The statement follows immediately from the fact that $x^{\frac{1}{r-1}}$ is concave on $[0, \infty)$ and $\deriv{}{x}x^{\frac{1}{r-1}} = \frac{1}{r-1} x^{-\frac{r-2}{r-1}}$.
\end{proof}

\begin{lem}
    \label{lem:auxlem3}
    Let $T, \lambda, M \in \RR_{> 0}$, $L, N \in \RR_{\ge 0}$, $r \in \NN$ with $r \ge 2$.
    Consider the system
    \begin{equation}
        \xi_{k+1} = \xi_k - T \lambda ( \spowf{\xi_{k+1} - \eta_k}{r-1}{r} - w_{k+1} )
    \end{equation}
    and define
    \begin{equation}
        \bar W_k =  \abs{\spowf{\xi_k}{r-1}{r} - w_k}.
    \end{equation}
    Then, the relations $\xi_{k}, \xi_{k+1} \in [-M, M]$, $\abs{\eta_k} \le N$ and $\abs{w_{k+1}-w_k} \le L$ imply
    \begin{equation}
        \bar W_{k+1}-  \bar W_k \le - \frac{r-1}{r} \frac{T \lambda( \bar W_{k+1} - 2^{\frac{1}{r}} N^{\frac{r-1}{r}})}{M ^{\frac{1}{r}}} + L
    \end{equation}
    whenever the right-hand side of this inequality is negative.
\end{lem}
\begin{proof}
    Due to symmetry reasons, consider without restriction of generality the case $\spowf{\xi_{k+1}}{r-1}{r} - w_{k+1} \ge 0$.
    Then, the relation $\spowf{\xi_{k+1} - \eta_k}{r-1}{r} - w_{k+1} = \bar W_{k+1} + \spowf{\xi_{k+1}-\eta_k}{r-1}{r} - \spowf{\xi_{k+1}}{r-1}{r}$ and Lemma~\ref{lem:auxlem} yield
    \begin{equation}
        \xi_{k} = \xi_{k+1} + T \lambda ( \spowf{\xi_{k+1} - \eta_k}{r-1}{r} - w_{k+1} ) 
\ge \xi_{k+1} + T \lambda \left( \bar W_{k+1} - 2^{\frac{1}{r}} N^{\frac{r-1}{r}} \right).
    \end{equation}
    Hence,
    \begin{align}
        \spowf{\xi_k}{r-1}{r} - w_k &\ge \spowf{\xi_{k+1} + T \lambda \left( \bar W_{k+1} - 2^{\frac{1}{r}} N^{\frac{r-1}{r}} \right)}{r-1}{r} -w_{k+1} - L \nonumber \\
        &\ge \spowf{\xi_{k+1}}{r-1}{r}  + \frac{r-1}{r} M^{-\frac{1}{r}} T \lambda \left( \bar W_{k+1} - 2^{\frac{1}{r}} N^{\frac{r-1}{r}} \right) - w_{k+1} - L \nonumber \\
        &= \bar W_{k+1} + \frac{r-1}{r} M^{-\frac{1}{r}} T \lambda \left( \bar W_{k+1} - 2^{\frac{1}{r}} N^{\frac{r-1}{r}} \right)  - L \ge \bar W_{k+1} \ge 0
    \end{align}
    by applying Lemma~\ref{lem:auxlem2}.
\end{proof}

\subsection{Proofs of the Main Lemmata}

\begin{proof}[Proof of Lemma~\ref{lem:invariant}]
    Suppose that $\x_{\bar K} \in \Omega$.
    With the purpose of obtaining a contradiction, assume first that $x_{1,\bar K+1} \ne \eta_{\bar K}$.
    Then,
    \begin{equation}
        (x_{1,\bar K+1} - \eta_{\bar K}) + L T^n \lambda_n \sign(x_{1,\bar K+1}- \eta_{\bar K} ) + \sum_{i=1}^{n-1} L^{\frac{i}{n}} T^{n-i} \lambda_i \spowf{x_{1,\bar K+1}- \eta_{\bar K} }{n-i}{n} - T^{n} \delta_{\bar K} = -\eta_{\bar K} + \sum_{i=1}^{n} T^{i-1} x_{i,\bar K}
    \end{equation}
    and noting that all but the last term on the left-hand side have the same sign yields the contradiction
    \begin{equation}
        L T^n (\lambda_n - 1) \le \abs{- \eta_{\bar K} + \sum_{i=1}^{n} T^{i-1} x_{i,\bar K}} \le N + \sum_{i=1}^{n} T^{i-1} \abs{x_{i,\bar K}} < L T^{n} (\lambda_n - 1) \left( 2^{-n} + \sum_{i=1}^{n} 2^{i-n-1}\right) = L T^{n} (\lambda_n -1).
    \end{equation}
    Hence, $x_{1,\bar K+1} = \eta_{\bar K}$.
    It will now be shown that also the inequalities $\abs{x_{i+1,\bar K+1}} \le 2^{i-n} L T^{n-i} (\lambda_n - 1)$ hold for $i = 1, \ldots, n-1$, i.e., that $\x_{\bar K+1} \in \Omega$.
    To see this by induction over $i$, note that $x_{1,\bar K+1} = \eta_{\bar K}$, $\x_{\bar K} \in \Omega$, and the induction assumption imply
    \begin{equation}
        |x_{i+1,\bar K+1}| = \abs{\frac{x_{i,\bar K+1}-x_{i,\bar K}}{T}} \le \frac{\abs{x_{i,\bar K+1}} + \abs{x_{i,\bar K}}}{T} \le \frac{2}{T} 2^{i-n-1}  L T^{n-i+1} (\lambda_n-1) = 2^{i-n} L T^{n-i} (\lambda_n - 1).
    \end{equation}
    Hence, $\x_{\bar K+1} \in \Omega$.
\end{proof}

\begin{proof}[Proof of Lemma~\ref{lem:deadbeat}]
    The statement is shown by induction over $i$.
    Clearly, it is true for $i = 1$, because $|x_{1,k}| = |\eta_{k-1}| \le N$ for $k \ge K + 1$.
    Suppose now that $|x_{i,k}| \le 2^{i-1} N$ for $k \ge \bar K + i$.
    Then, \eqref{eq:diff:errorsystem:i} with $x_{1,k+1} -\eta_k = 0$ yields $x_{i+1,k+1} = (x_{i,k+1}-x_{i,k})/T$.
    Hence,
    \begin{equation}
        |x_{i+1,\bar K+i+1}| = \frac{|x_{i,\bar K+i+1} - x_{i,\bar K+i}|}{T} \le \frac{|x_{i,\bar K+i+1}| + |x_{i,\bar K+i}|}{T} \le \frac{2}{T} \frac{2^{i-1} N}{T^{i-1}} = \frac{2^i N}{T^i}
    \end{equation}
    is obtained, proving the claim.
\end{proof}

\begin{proof}[Proof of Lemma~\ref{lem:errorbound}]
    Consider the Newton polynomials $w_j : \RR \to \RR$ defined as
    \begin{equation}
        w_j(x) = \frac{1}{j!} \prod_{p=0}^{j-1} (x+p) = \frac{1}{j!} x (x+1) \ldots (x+j-1).
    \end{equation}
    It will first be shown that the constants $c_{i,j}$ satisfy $c_{i,j} = w_j^{(i)}(0)$ for all $i, j \in \NN_0$.
    To see this, note that
    $
        w_{j}(x) = w_{j-1}(x) \frac{x+j-1}{j}
        $
    holds for $j \ge 1$ with $w_0(x) = 1$ and hence, by induction over $i$,
    \begin{equation}
w^{(i)}_{j}(x) = w_{j-1}^{(i)}(x) \frac{x+j-1}{j} + w_{j-1}^{(i-1)}(x) \frac{i}{j}
    \end{equation}
    is obtained for all $i,j \in \NN$.
    Evaluation at $x = 0$ then yields the recursion \eqref{eq:cij} with initial conditions following from $w_0(x) = 1$ and $w_j(0) = 0$ for $j \ge 1$.

    For arbitrary $k \in \NN_0$ apply Newton's interpolation formula with the Newton polynomials $w_j$ and the divided differences $g^{f}_{j+1,k+1}$ to define the polynomial
    \begin{equation}
        p(t) = \sum_{j=0}^{m} T^{j} w_{j}(t/T - k) g^{f}_{j+1,k+1}
    \end{equation}
    interpolating $\bar f(t)$ defined in \eqref{eq:fbar} at $t = kT, (k-1)T, \ldots, (k-m)T$.
    According to Shadrin et al.\cite[Theorem~A]{shadrin1995error}, the approximation error of its derivative
    \begin{equation}
        p^{(i)}(kT) =\sum_{j=i}^{m} T^{j-i} w_j^{(i)}(0) g^{f}_{j+1,k+1} =\sum_{j=i}^{m} T^{j-i}c_{i,j}  g^{f}_{j+1,k+1}
    \end{equation}
    is then bounded from above by
    \begin{equation}
        |p^{(i)}(kT) - f^{(i)}(kT)| \le MT^{n-i} w_n^{(i)}(0) = c_{i,n} M T^{n-i}.
    \end{equation}
    for $i = 1, \ldots, m$.
    This proves relation \eqref{eq:lem:errorbound}.

    To show the claimed equality, let $f(t) = \frac{M t^{m+1}}{(m+1)!}$, fix an integer $k \ge  m$ and define functions $h_1, h_2$ as $h_1(t) = M T^{n} w_n(t/T-k)$ and $h_2 = f - h_1$.
    Then, $h_2$ is a polynomial of degree $m$, i.e., $h_2  \in \FLz{m}$, which allows to conclude from \eqref{eq:lem:errorbound} with $M = 0$ that
    \begin{equation}
        h_2^{(i)}(kT) = \sum_{j = i}^{m} T^{j-i} c_{i,j} g^{h_2}_{j+1,k+1}
    \end{equation}
    Moreover, it is easy to verify that $h_1(kT) = h_1( (k-1) T) = \ldots = h( (k-m) T) = 0$, and hence the corresponding divided differences satisfy $g^{h_1}_{j,k+1} = 0$ for $j = 1, \ldots, n$.
    Since the divided differences $g^f_{i,k}$ are linear in $f$, i.e., $g_{i,k}^{h_1} + g_{i,k}^{h_2} = g_{i,k}^{h_1+h_2}$, the claimed equality
    \begin{equation}
        |f^{(i)}(kT) - \sum_{j=i}^{m} T^{j-i} c_{i,j} g^{f}_{j+1,k+1}| = |f^{(i)}(kT) - \sum_{j=i}^{m} T^{j-i} c_{i,j} g^{h_2}_{j+1,k+1}|  = |h_1^{(i)}(kT)| = M T^{n-i} |w_n^{(i)}(0)| = c_{i,n} M T^{n-i}
    \end{equation}
    follows.
\end{proof}

\begin{proof}[Proof of Lemma~\ref{lem:Vbound}]
    The inequality $\abs{\xi_j} \le \beta_j^j V_j(\vxi)^{j}$ is proven by induction over $j$.
    For $j = 1$, $\abs{\xi_1} \le \beta_1 V_1(\vxi)$ is obvious.
    Let $r \in \NN$, $r \ge 2$ and suppose that the statement is true for $j = r-1$.
    Then,
    \begin{align}
        \abs{\xi_{r}} &\le \left(\abs{\xi_{r-1}} + \abs{\spowf{\xi_r}{r-1}{r} - \xi_{r-1}}\right)^{\frac{r}{r-1}} \nonumber \\
        &\le (\beta_{r-1}^{r-1} V_{r-1}(\vxi)^{r-1} + \alpha_{r} V_r(\vxi)^{r-1})^{\frac{r}{r-1}} \nonumber \\
        &\le (\beta_{r-1}^{r-1} + \alpha_{r} )^{\frac{r}{r-1}} V_r(\vxi)^{r} = \beta_r^{r} V_r(\vxi)^{r}
    \end{align}
    proving the claim.
    The remaining statement of the lemma follows from the fact that $V_n(\vxi) \ge V_j(\vxi)$ holds for all $j$ and all $\vxi \in \RR^n$ by definition \eqref{eq:V}.
\end{proof}

\newcommand{\Vjk}[2]{V_{#1,#2}}
\begin{proof}[Proof of Lemma~\ref{lem:main}]
    The statement is proven by induction over $j$.
    For simplicity, the abbreviations $V_{j,k} = V_j(\vxi_k)$ and $V_{j,k+1} = V_j(\vxi_{k+1})$ are used throughout the proof.
    For $j = 1$, the inequalities $\abs{\xi_{1,k+1}} = \Vjk{1}{k+1} > \gamma_1 N/(2L) = N/L$ and $\abs{\eta_{2,k+1}} \le N/L$ imply that $\xi_{1,k+1}$ has the same sign as $\xi_{1,k+1}+\eta_{2,k+1}$, i.e., $\spow{\xi_{1,k+1}+\eta_{2,k+1}}{0} = \spow{\xi_{1,k+1}}{0} = \{ \sign(\xi_{1,k+1}) \}$, and hence
    \begin{equation}
        \Vjk{1}{k} = \abs{\xi_{1,k}} = \abs{\xi_{1,k+1} + T \kappa_1 \bigl( \sign(\xi_{1,k+1}) - \xi_{0,k+1}\bigr)} \ge \abs{\xi_{1,k+1}} + T \kappa_1 (1 - \kappa_0) = \Vjk{1}{k+1} + \epsilon_1
    \end{equation}
i.e., $\Vjk{1}{k+1} \le \Vjk{1}{k} - \epsilon_1$ holds with $\epsilon_1 = T\kappa_1 (1-\kappa_0) > 0$.
    Let now $r \in \NN$, $r\ge 2$ and suppose that the statment is true for $j = r-1 \ge 1$.
    It will be shown that it is then true also for $j = r$.
    Define $\tilde N_{r+1} := \frac{N}{L}$ and
    \begin{equation}
        \label{eq:tildeNi}
        \tilde N_i = \alpha_{i} \Vjk{i}{k+1}^{i-1} + 2 \left(\tilde N_{i+1}/2\right)^\frac{i-1}{i}
    \end{equation}
    for $i = r$ and $i = r-1$.
    Noting that $|\eta_{r+1,k+1}| \le \tilde N_{r+1}$ by assumption, application of Lemma~\ref{lem:auxlem} then yields
    \begin{align}
        \label{eq:tildeNm}
        \abs{\eta_{r,k+1}} &= \abs{\spowf{\xi_{r,k+1}  + \eta_{r+1,k+1}}{r-1}{r}- \xi_{r-1,k+1}}
        \le \abs{\spowf{\xi_{r,k+1}}{r-1}{r} - \xi_{r-1,k+1}} + 2^{\frac{1}{r}} |\eta_{r+1,k+1}|^{\frac{r-1}{r}}  \nonumber \\
        &\le \alpha_{r} \Vjk{r}{k+1}^{r-1} + 2 \left(\frac{\abs{\eta_{r+1,k+1}}}{2}\right)^\frac{r-1}{r}
        \le \alpha_{r} \Vjk{r}{k+1}^{r-1} + 2 \left(\frac{\tilde N_{r+1}}{2}\right)^\frac{r-1}{r} = \tilde N_r.
    \end{align}
    Analogously,
    \begin{align}
        \label{eq:tildeNmm1}
        \abs{\eta_{r-1,k+1}} \le \alpha_{r-1} \Vjk{r-1}{k+1}^{r-2} + 2^{\frac{1}{r-1}} \abs{\eta_{r,k+1}}^\frac{r-2}{r-1}
        \le \alpha_{r-1} \Vjk{r-1}{k+1}^{r-2} + 2 (\tilde{N}_{r}/2)^\frac{r-2}{r-1} = \tilde N_{r-1},
    \end{align}
    is obtained.
Note that $|\eta_{r-1,k+1}| \le \tilde N_{r-1}$ holds even in the case $r = 2$ by virtue of the convention $\alpha_1 = 0$, because $\tilde N_{1}$ as defined in \eqref{eq:tildeNi} is given by $\tilde N_1 = \alpha_1 V_{1,k+1}^0 + 2 (\tilde{N}_{2}/2)^{0} = 2$ and
    \begin{equation}
        \eta_{1,k+1} \in \spow{\xi_{1,k+1} + \eta_{2,k+1}}{0} -  \xi_{0,k+1} \subseteq [-1-\kappa_0, 1+\kappa_0] \subset [-2,2] = [-\tilde N_1, \tilde N_1].
    \end{equation}
    With $V_r$ being given by
    \begin{equation}
        V_{r}(\vxi) = \max\left\{ V_{r-1}(\vxi), \alpha_{r}^{-\frac{1}{r-1}} \abs{\spowf{\xi_{r}}{r-1}{r} - \xi_{r-1}}^{\frac{1}{r-1}}  \right\}
    \end{equation}
    according to \eqref{eq:V}, distinguish now the two cases $\Vjk{r}{k+1} = \Vjk{r-1}{k+1}$ and $\Vjk{r}{k+1} = \alpha_{r}^{-\frac{1}{r-1}} \abs{\spowf{\xi_{r,k+1}}{r-1}{r} - \xi_{r-1,k+1}}^{\frac{1}{r-1}}$.

    In the first case, $\Vjk{r}{k+1} > \gamma_r \sqrt[r]{\tilde N_{r+1}/2}$ and the recursion \eqref{eq:recursion}, specifically $\gamma_{r}^{r-1} = \frac{2 \gamma_{r-1}^{r-1}}{2 - \alpha_r \gamma_{r-1}^{r-1}}$, imply
    \begin{equation}
        (2 - \gamma_{r-1}^{r-1} \alpha_{r} )\Vjk{r}{k+1}^{r-1} > 2 \gamma_{r-1}^{r-1} (\tilde N_{r+1}/2)^{\frac{r-1}{r}}
    \end{equation}
    which is equivalent to 
    \begin{equation}
        \Vjk{r}{k+1}^{r-1} > \frac{\gamma_{r-1}^{r-1}}{2} \left(\alpha_{r} \Vjk{r}{k+1}^{r-1} + 2 (\tilde N_{r+1}/2)^\frac{r-1}{r}\right) = \gamma_{r-1}^{r-1} \frac{\tilde N_r}{2}
    \end{equation}
    and yields $\Vjk{r-1}{k+1} = \Vjk{r}{k+1}  > \gamma_{r-1} \sqrt[r-1]{\tilde N_r/2}$.
    Hence, by using the induction assumption with $N = L \tilde N_r$,
    \begin{equation}
        \Vjk{r}{k+1} = \Vjk{r-1}{k+1} \le \Vjk{r-1}{k} - \epsilon_{r-1} \le \Vjk{r}{k} - \epsilon_{r-1}
    \end{equation}
    holds for some $\epsilon_{r-1} > 0$.

    In the second case, consider the difference equation
    \begin{equation}
        \xi_{r,k+1} = \xi_{r,k} - T \kappa_r  \eta_{r,k+1} = \xi_{r,k} - T \kappa_r ( \spowf{\xi_{r,k+1} + \eta_{r+1,k+1}}{r-1}{r} - \xi_{r-1,k+1} )
    \end{equation}
    according to \eqref{eq:diff:errorsystem:transformed2} and define $\bar W_k = |\spowf{\xi_{r,k}}{r-1}{r} - \xi_{r-1,k}|$, $\bar W_{k+1} = |\spowf{\xi_{r,k+1}}{r-1}{r} - \xi_{r-1,k+1}|$.
    Furthermore, define $W_k = \alpha_{r} \Vjk{r}{k}^{r-1}$ and $W_{k+1} = \alpha_{r} \Vjk{r}{k+1}^{r-1}$, noting that $W_k \ge \bar W_k$ due to \eqref{eq:V} and $W_{k+1} = \bar W_{k+1}$ by virtue of the considered case.
    Now, note that $\xi_{r,k+1}, \xi_{r,k} \in [-M,M]$ with the abbreviation $M = \beta_r^r \max\{ \Vjk{r}{k+1}^r, \Vjk{r}{k}^r \} = \beta_r^r \alpha_{r}^{-\frac{r}{r-1}} \max\{ W_{k+1}^{\frac{1}{r-1}}, W_k^{\frac{1}{r-1}} \}$ due to Lemma~\ref{lem:Vbound}.
    Since $\abs{\eta_{r+1,k+1}} \le \tilde N_{r+1}$ and $|\xi_{r-1,k+1}-\xi_{r-1,k}| = |T \kappa_{r-1} \eta_{r-1,k+1}| \le T\kappa_{r-1} \tilde N_{r-1}$ hold according to \eqref{eq:tildeNm} and \eqref{eq:tildeNmm1}, now apply Lemma~\ref{lem:auxlem3} to obtain
    \begin{align}
        W_{k+1} - W_k &\le \bar W_{k+1} - \bar W_k \le -\frac{r-1}{r}\cdot \frac{T \kappa_r (\bar W_{k+1} - 2^{\frac{1}{r}} \tilde N_{r+1}^{\frac{r-1}{r}})}{M^{\frac{1}{r}}}  + T \kappa_{r-1} \tilde N_{r-1} \\
        &= -\frac{r-1}{r}\cdot \frac{T \kappa_r \alpha_{r}^{\frac{1}{r-1}} (W_{k+1} - 2^{\frac{1}{r}} \tilde N_{r+1}^{\frac{r-1}{r}})}{\beta_r \max\{ W_{k+1}^{\frac{1}{r-1}}, W_k^{\frac{1}{r-1}} \}}  + T \kappa_{r-1} \tilde N_{r-1},
    \end{align}
    provided that the right-hand side of this inequality is positive, which will be shown to follow from condition \eqref{eq:cond:aux} later on.
    Since $W_{k+1}$ is positive by assumption, Lemma~\ref{lem:auxlem4} may be applied to obtain
\begin{align}
        \label{eq:Wincrement}
        W_{k+1}^{\frac{1}{r-1}} &\le W_k^{\frac{1}{r-1}} -\frac{T \kappa_r }{r}\frac{\alpha_{r}^{\frac{1}{r-1}} (W_{k+1} - 2^{\frac{1}{r}} \tilde N_{r+1}^{\frac{r-1}{r}})}{\beta_r \max\{ W_{k+1}^{\frac{1}{r-1}}, W_k^{\frac{1}{r-1}} \} W_k^{\frac{r-2}{r-1}}}  + \frac{T \kappa_{r-1}}{r-1} \frac{\tilde N_{r-1}}{W_k^{\frac{r-2}{r-1}}} \nonumber \\
        &= W_k^{\frac{1}{r-1}} -\frac{T \kappa_r }{r}\frac{\alpha_{r}^{\frac{1}{r-1}} (W_{k+1} - 2 (\tilde N_{r+1}/2)^{\frac{r-1}{r}})}{\beta_r \max\{ W_{k+1}^{\frac{1}{r-1}}, W_k^{\frac{1}{r-1}} \} W_k^{\frac{r-2}{r-1}}}
        + \frac{T \kappa_{r-1}}{r-1} \frac{\alpha_{r-1} \Vjk{r-1}{k+1}^{r-2} + 2^{\frac{1}{r-1}} \left( W_{k+1} + 2 (\tilde N_{r+1}/2)^\frac{r-1}{r} \right)^{\frac{r-2}{r-1}}}{W_k^{\frac{r-2}{r-1}}} \nonumber \\
        &\le W_k^{\frac{1}{r-1}} -\frac{T \kappa_r }{r}\frac{\alpha_{r}^{\frac{1}{r-1}} (W_{k+1} - 2 (\tilde N_{r+1}/2)^{\frac{r-1}{r}})}{\beta_r \max\{ W_{k+1}^{\frac{1}{r-1}}, W_k^{\frac{1}{r-1}} \} W_k^{\frac{r-2}{r-1}}}
        + \frac{T \kappa_{r-1}}{r-1} \frac{\alpha_{r-1} \alpha_{r}^{-\frac{r-2}{r-1}} W_{k+1}^{\frac{r-2}{r-1}} + 2^{\frac{1}{r-1}} \left( W_{k+1} + 2 (\tilde N_{r+1}/2)^\frac{r-1}{r} \right)^{\frac{r-2}{r-1}}}{W_k^{\frac{r-2}{r-1}}}.
    \end{align}
    It will be shown that for every $N > 0$, i.e., for every $\tilde N_{r+1} > 0$, there exists $\epsilon > 0$ such that $W_{k+1} > \alpha_{r} \gamma_{r}^{r-1} (\tilde N_{r+1}/2)^{\frac{r-1}{r}}$, i.e., $V_{k+1} > \gamma_r \sqrt[r]{\tilde N_{r+1}/2}$, implies $W_{k+1}^{\frac{1}{r-1}} \le W_k^{\frac{1}{r-1}} - \epsilon$, which yields the claimed inequality $\Vjk{r}{k+1} \le \Vjk{r}{k} - \epsilon_r$ with $\epsilon_r = \alpha_{r}^{-\frac{1}{r-1}} \epsilon$.
    To that end, assume to the contrary that $W_k^{\frac{1}{r-1}} < W_{k+1}^{\frac{1}{r-1}} + \epsilon$ for all $\epsilon > 0$.
    Then, also $\max\{ W_{k+1}^{\frac{1}{r-1}}, W_k^{\frac{1}{r-1}} \} \le W_{k+1}^\frac{1}{r-1} + \epsilon$ holds in \eqref{eq:Wincrement}, and hence
    \begin{equation}
        W_{k+1}^{\frac{1}{r-1}} \le W_k^{\frac{1}{r-1}} - T h\left(\epsilon, W_{k+1} (\tilde N_{r+1}/2)^{-\frac{r-1}{r}}\right)
    \end{equation}
    holds with
    \begin{equation}
        h(\epsilon, W) = \frac{\kappa_r}{r}\frac{\alpha_{r}^{\frac{1}{r-1}} (W - 2)}{\beta_r (W^{\frac{1}{r-1}} + \epsilon (\tilde N_{r+1}/2)^{-\frac{1}{r}})^{r-1}} -\frac{\kappa_{r-1}}{r-1} \frac{\alpha_{r-1} \alpha_{r}^{-\frac{r-2}{r-1}} W^\frac{r-2}{r-1} + 2^{\frac{1}{r-1}} \left( W + 2 \right)^{\frac{r-2}{r-1}}}{(W^{\frac{1}{r-1}} + \epsilon (\tilde N_{r+1}/2)^{-\frac{1}{r}})^{r-2}}
    \end{equation}
    whenever $h(\epsilon, W_{k+1} (\tilde N_{r+1}/2)^{-\frac{r-1}{r}}) \ge 0$.
    Note that $h(0,W)$ is strictly increasing and condition 
\eqref{eq:cond:aux} with $p = r$ implies
    \begin{align}
        h(0,\alpha_{r} \gamma_r^{r-1}) &= \frac{\kappa_r}{r}\frac{\alpha_{r}^{\frac{1}{r-1}} (\alpha_{r} \gamma_r^{r-1} - 2)}{\beta_r \alpha_{r} \gamma_r^{r-1}} -\frac{\kappa_{r-1}}{r-1} \frac{\alpha_{r-1} \gamma_r^{r-2} + 2^{\frac{1}{r-1}} \left( \alpha_{r} \gamma_r^{r-1} + 2 \right)^{\frac{r-2}{r-1}}}{\alpha_{r}^{\frac{r-2}{r-1}} \gamma_r^{r-2}} \nonumber \\
        &= \frac{1}{\alpha_r^{\frac{r-2}{r-1}} \gamma_r^{r-1}} \left[ \frac{\kappa_r}{r}\frac{\alpha_{r} \gamma_r^{r-1} - 2}{\beta_r} -\frac{\kappa_{r-1}}{r-1} \left( \alpha_{r-1} \gamma_r^{r-1} + 2^{\frac{1}{r-1}} \gamma_r \left( \alpha_{r} \gamma_r^{r-1} + 2 \right)^{\frac{r-2}{r-1}}\right) \right]  > 0.
    \end{align}
    Since $h$ is also continuous with respect to $\epsilon$, there exists $\epsilon > 0$ such that $h(\epsilon, W) > \epsilon T^{-1}$ for all $W > \alpha_{r}\gamma_r^{r-1}$, yielding the contradiction $W_{k}^{\frac{1}{r-1}} \ge W_{k+1}^{\frac{1}{r-1}} + \epsilon$.
\end{proof}

\bibliography{literature}

\end{document}